\documentclass[11pt,a4paper,reqno]{amsart}
\usepackage{amsmath,amssymb,amsfonts,epsfig,mathrsfs}
\usepackage[T1]{fontenc}

\usepackage{color}
\usepackage{array}
\usepackage{amsthm}
\usepackage{amstext}
\usepackage{graphicx}
\usepackage{setspace}
\usepackage[margin=2.5cm]{geometry}
\usepackage{bbm}
\usepackage{color}
\usepackage{enumitem}
\usepackage{undertilde}
\usepackage{amscd,psfrag}
\usepackage{yhmath}
\usepackage[mathscr]{eucal}

\usepackage{slashed}

\makeatletter
\pdfpageheight\paperheight
\pdfpagewidth\paperwidth

\usepackage{epstopdf}
\usepackage{chngcntr}
\counterwithin{figure}{section}
\usepackage{mathrsfs}

\usepackage{indentfirst}	

\usepackage[normalem]{ulem}
\theoremstyle{plain}

\newtheorem{definition}{Definition}[section]
\newtheorem{theorem}[definition]{Theorem}
\newtheorem*{theorem*}{Theorem}

\newtheorem{remark}[definition]{Remark}

\newtheorem*{remark*}{Remark}
\newtheorem*{sideremark*}{Side Remark}

\newtheorem*{claim*}{Claim}
\newtheorem*{q*}{Question}
\newtheorem{lemma}[definition]{Lemma}
\newtheorem{corollary}[definition]{Corollary}
\newtheorem*{corollary*}{Corollary}

\newtheorem{proposition}[definition]{Proposition}

\newcommand{\R}{\mathbb{R}}
\newcommand{\na}{\nabla}
\newcommand{\lie}{\mathcal{L}}
\newcommand{\dd}{{\rm d}}
\newcommand{\dx}{\,\dd x}
\newcommand{\dt}{\,\dd t}
\newcommand{\oo}{\omega}
\newcommand{\p}{\partial}
\newcommand{\n}{{\mathbf{n}}}
\newcommand{\po}{{\partial \Omega}}

\newcommand{\rot}{{\mathcal{R}}}
\newcommand{\two}{{\rm II}}
\newcommand{\proj}{{{\mathbb{P}}_{\infty}}}
\newcommand{\curl}{{\bf curl}}
\newcommand{\ddh}{{\, {\rm d}\mathcal{H}^2}}
\newcommand{\nr}{\na^{(r-3)}}
\newcommand{\e}{\epsilon}
\newcommand{\emb}{\hookrightarrow}
\newcommand{\q}{\mathbf{q}}
\newcommand{\shape}{\mathcal{S}}
\newcommand{\vr}{{\mathbf{V}_r^{\nu}}}
\newcommand{\vv}{{v^\nu}}
\newcommand{\V}{\mathcal{V}}
\newcommand{\ttau}{{\boldsymbol \tau}}

\numberwithin{equation}{section}
\numberwithin{figure}{section}

\title{The Inviscid Limit of the Navier-Stokes Equations \\ with Kinematic and Navier Boundary Conditions}
\author{Gui-Qiang G. Chen}
\address{Gui-Qiang G. Chen: Mathematical Institute,\
 University of Oxford, Oxford, OX2 6GG, UK;
School of Mathematics Sciences,
Fudan University, Shanghai 200433, China;
AMSS \& UCAS,
Chinese Academy of Sciences, Beijing 100190, China}
\email{\texttt{chengq@maths.ox.ac.uk}}
\author{Siran Li}
\address{Siran Li:
Department of Mathematics, Rice University, MS 136, P.O. Box 1892, Houston,
Texas, 77251-1892, USA; Department of Mathematics, McGill University, Burnside Hall, 805
Sherbrooke Street West, Montreal, Quebec, H3A 0B9, Canada.}
\email{\texttt{Siran.Li@rice.edu}}
\author{Zhongmin Qian}
\address{Zhongmin Qian: Mathematical Institute,\
 University of Oxford, Oxford, OX2 6GG, UK}
\email{\texttt{qianz@maths.ox.ac.uk}}

\date{\today}

\begin{document}

\begin{abstract}
We are concerned with the inviscid limit of the Navier-Stokes equations on bounded regular domains in $\R^3$
with the kinematic and Navier boundary conditions.
We first establish the existence and uniqueness of strong solutions in
the class $C([0,T_\star); H^r(\Omega; \R^3)) \cap C^1([0,T_\star); H^{r-2}(\Omega;\R^3))$
with some $T_\star>0$ for
the initial-boundary value problem with the kinematic and Navier boundary conditions on $\partial \Omega$
and divergence-free initial data in the Sobolev space $H^r(\Omega; \R^3)$ for $r\geq 2$.
Then, for the strong solution with $H^{r+1}$--regularity in the spatial variables,
we establish the inviscid limit in $H^r(\Omega; \R^3)$ uniformly on $[0,T_\star)$ for $r > \frac{5}{2}$.
This shows that the boundary layers do not develop up to the highest order Sobolev norm in $H^{r}(\Omega;\R^3)$
in the inviscid limit.
Furthermore, we present an intrinsic geometric proof for the failure of the strong inviscid limit
under a non-Navier slip-type boundary condition.
\end{abstract}
\maketitle

\section{Introduction}

We are interested in the analysis of strong solutions in the Sobolev spaces $H^r$ of the incompressible Navier-Stokes equations
with positive viscosity
coefficient $\nu>0$ in a bounded regular domain $\Omega\subset \R^3$ subject to the kinematic and Navier boundary conditions on $\partial\Omega$
and the divergence-free initial data at $t=0$,
and their convergence to the corresponding strong solution of the Euler equations in the inviscid limit as $\nu\to 0$.
One of our main motivations for such an analysis is to examine whether the boundary layers would develop in some high-order Sobolev norm
in the inviscid limit.

We assume that the boundary, $\po$, of domain $\Omega$ is an embedded oriented $2$-dimensional (2-D) manifold,
{\it i.e.} a regular surface.
The incompressible Navier-Stokes equations in $[0,T]\times\Omega$ take the following form:
\begin{equation}\label{NS equation}
\begin{cases}
\p_t u^\nu + (u^\nu \cdot \na) u^\nu  + \na p^\nu =\nu \Delta u^\nu,\\
\na \cdot u^\nu = 0.
\end{cases}
\end{equation}
In \eqref{NS equation}, the vector field $u^\nu: \Omega \rightarrow \R^3$ is the velocity
of the fluid and the scalar field $p^\nu: \Omega \rightarrow \R$ is the pressure,
both of which depend on the viscosity constant $\nu>0$.
The divergence-free condition of $u^\nu$ describes the incompressibility of the fluid.
The existence, uniqueness, and regularity of weak and strong solutions of the Navier-Stokes equations \eqref{NS equation}
are an important research topic in nonlinear PDEs and mathematical hydrodynamics;
{\it cf.} \cite{kato, kato-ponce, majda, temam, seregin} and the references cited therein.
In this paper, we focus on the Navier-Stokes equations \eqref{NS equation} in a general bounded regular
domain $\Omega$,
for which the geometry of $\Omega$ plays an important role in our analysis.

Consider the initial condition:
\begin{equation}\label{IDC}
u^\nu|_{t=0} = u_0 \qquad \text{ on }\Omega,
\end{equation}
where $u_0$ satisfies the compatibility condition:  $\nabla \cdot u_0=0$ in $\Omega$.

The {\em kinematic boundary condition} is
\begin{equation}\label{KBC}
u^\nu \cdot \n = 0 \qquad \text{ on } \p\Omega \times [0,T],
\end{equation}
{\it i.e.} the normal component
of the velocity on the boundary vanishes.

The {\em Navier boundary condition} is imposed as:
\begin{equation}\label{classical Navier boundary condition}
u \cdot \ttau = -2\zeta \,  \mathbb{D}u (\tau, \n) \qquad \text{ on }  \p\Omega\times[0,T],
\end{equation}
for any $\ttau \in T(\p\Omega)$, where the rate-of-strain tensor is the $3\times 3$ matrix defined by
\begin{equation}
\mathbb{D}u := \frac{1}{2}\big(\na u + (\na u)^\top\big),
\end{equation}
$\mathbb{D}u (\tau, \n):=\ttau^\top  \mathbb{D}u\,\n$,
and constant $\zeta >0$ is known as the {\em slip length} of the fluid.

Traditionally, the Navier-Stokes equations \eqref{NS equation} have been studied with the no-slip condition,
{\it i.e.} the Dirichlet boundary condition $u = 0$ on $\po$.
However, this does not always match with the experimental data; {\it cf.} \cite{einzel, qian}.
First proposed by Navier \cite{navier} in 1816,
the Navier boundary condition \eqref{classical Navier boundary condition}
requires that the tangential component of the velocity field is proportional to that of the normal
vector field of the Cauchy stress tensor.
The proportionality constant $\zeta>0$ is known as the {\em slip length}.
Physically, the Navier boundary condition \eqref{classical Navier boundary condition}
can be induced by the effects of free capillary boundaries,
perforated boundaries, or the exterior electric fields;
{\it cf.} Achdou-Pironneau-Valentin \cite{achdou}, B\"{a}nsch \cite{bansch},
Beavers-Joseph \cite{beavers}, Einzel-Panzer-Liu \cite{einzel},
Maxwell \cite{maxwell}, J\"{a}ger-Mikeli\u{c} \cite{jager1, jager2},
and the references cited therein.

To analyze the initial-boundary value problem \eqref{IDC}--\eqref{classical Navier boundary condition}
for the Navier-Stokes equations \eqref{NS equation},
we adopt an equivalent geometric formulation, as shown in Chen-Qian \cite{cq2}, for the boundary conditions on $\po \times [0,T]$:
\begin{equation}\label{boundary conditions}
\begin{cases}
u^\nu \cdot \n  = 0,\\
\omega^\nu \cdot \ttau = - \frac{1}{\zeta} (\rot u^\nu)\cdot \ttau
   + 2 \rot(\shape(u^\nu)) \cdot \ttau \qquad \text{ on } \po \times [0,T],
\end{cases}
\end{equation}
where
\begin{equation}
\omega^\nu := \na \times u^\nu
\end{equation}
is the vorticity of the fluid,  $\ttau\in T(\po)$ is an arbitrary tangential vector field
on boundary $\po$, $\shape$ is the shape operator of surface $\p\Omega$,
and $\rot$ is the operator corresponds to the left multiplication by the matrix in the local coordinate on $\p\Omega$:
\begin{equation}
\rot =  \begin{bmatrix}
0 & -1\\
1 & 0
\end{bmatrix},
\end{equation}
{\it i.e.} the anti-clockwise rotation by $\frac{\pi}{2}$.
In fact,  $\rot$ can be identified with the Hodge star operator $\ast$ defined for the differential forms on $\R^2$:
For a 2-D vector field $V=(V^1, V^2)^\top$,
\begin{equation}
\rot V = (\ast ( V^\sharp))^\flat = (-V^2, V^1)^\top,
\end{equation}
in which $\sharp$ is the canonical isomorphism between vector fields and differential $1$-forms,
and $\flat$ is its inverse.
The second equation in \eqref{boundary conditions} ({\it i.e.} the Navier boundary condition
in the geometric formulation) has the vorticity on the left-hand side,
but it involves only the zero-th order operations on the velocity  on the right-hand side.

The problem of inviscid limits has been a central topic in mathematical
hydrodynamics ({\it cf.} Constantin \cite{c2}).
In 1975, Swann \cite{swann} proved that, for $\Omega = \R^3$,
when the initial vorticity is in $H^{3+\delta}$, divergence--free, and vanishing at spatial infinity,
and the right--hand side of the vorticity equation lies in $C([0,T); H^2(\R^3))$ for some small $T$,
then the initial-boundary value problem for the Navier--Stokes equations
with zero boundary condition has a unique strong solution,
and the vanishing viscosity limit holds in $L^6 \cap \dot{H}^1$.
In 1986, Constantin \cite{c1} showed that, for $\Omega = \R^3$,
if the Cauchy problem for the Euler equations with initial data $v_0 \in H^{m+2}(\R^3)$ for $m \geq 3$
has a strong solution in $X=C([0,T]; H^m(\R^3))$ up to time $T$,
then there exists $\nu_\star=\nu_\star(T, v_0)$ such that the Cauchy problem for
the corresponding Navier-Stokes
equations for any $\nu \leq \nu_\star$ also has a strong solution in $X$,
and the vanishing viscosity limit holds in $H^m$.
In fact, for $\Omega=\R^d$ for $d=2$ or $3$,
for any $s>\frac{d}{2}+1$ and initial data $v_0 \in H^s$,
the convergence can be obtained in the $H^s$--norm;
{\it cf.} Masmoudi \cite{masmoudi}.
Moreover, in Constantin-Wu \cite{c4}, the vanishing viscosity limits
were also proved on $\Omega = \R^2$ for the initial vorticity in $L^1(\R^2) \cap L^\infty_c(\R^2)$.

On the other hand, in the case that $\Omega$ is a bounded domain with boundary,
and the Navier-Stokes equations are
equipped with the Dirichlet boundary condition,
the vanishing viscosity limit fails in general:
This is due to the formation of {\em boundary layers},
in which the Prandtl equations serve as a candidate for matching the Navier-Stokes and Euler equations;
see {\it e.g.}, Alexandre-Wang-Xu-Yang \cite{alexandre}, G\'{e}rard-Varet-Dormy \cite{gv},
and the references cited therein.
In contrast, when the Navier and kinematic boundary conditions are imposed
to the Navier-Stokes equations, the vanishing viscosity limit can be established
in the affirmative. In 2007, Xiao-Xin  \cite{xin1} proved that,
for the initial data in $H^3$ on a 3-D flat domain, there exists $T_\star >0$ such that
the vanishing viscosity limit holds in $C([0,T_0]; H^2)\cap L^p(0, T_0; H^3)$,
for all $1 \leq p <\infty$.
Various convergence results of this kind for a non-Navier ``slip-type boundary condition''
(first proposed by Bardos \cite{bardos}, which agrees with the Navier condition if and only if the domain is a part of
the flat half--space) have been established, in $W^{k,p}$, $H^s$, or $L^p$ spaces and on 2-D  or 3-D spatial domains;
{\it cf.} Xiao-Xin  \cite{xin1}, Beir\~{a}o da Veiga-Crispo \cite{v1, v2}, Bellout-Neustupa-Penel \cite{bellout},
 Berselli-Spirito \cite{berselli}, Chen-Osborne-Qian \cite{cq1},
 Clopeau-Mikeli\u{c}-Robert \cite{clopeau},  Kelliher \cite{kelliher},
 Wang-Xin-Zang \cite{xin2}, Zhong \cite{zhong}, and the references cited therein.

Furthermore, for the Navier boundary conditions,
Chen-Qian \cite{cq2} and Iftimie-Planas \cite{iftimie1} obtained the vanishing viscosity limit
in $L^\infty_t L^2_x$ on smooth domains $\Omega \subset \R^3$ and $\R^d, d \geq 2$,
provided that strong solutions exist in $H^2$ and $H^{d/2+1+\e}$, respectively;
see also the related results by Ifitimie-Raugel-Sell \cite{iftimie2} on a 3-D thin domain
and by Lopes Filho-Nussenzveig Lopes-Planas \cite{lopes} on 2-D domains,
and the recent results by Drivas-Nguyen \cite{DN}.
In addition, by computations in local coordinates, Neustupa-Penel \cite{neustupa, neustupa2} proved
the convergence in $L^\infty(0, T_\star; H^1) \cap L^2(0, T_\star; H^2)$, provided that the initial
data is in $H^4$, where $T_\star>0$ is a constant depending only on $\Omega$ and the initial data.
Moreover, using the geometric vector field approach, Masmoudi-Rousset \cite{masmoudi-rousset}
established the existence of strong solutions in $L^\infty(0, T; E^m(\Omega;\R^3))\cap L^2(0, T; H^{m+1}(\Omega;\R^3))$
for $m >6$ and the inviscid limit in $L^\infty_tL^2_x$, where the anisotropic co-normal Sobolev space
$E^m$ is given by $E^m := \{u \in H^m_{\rm co}\,:\, \na u \in H^{m-1}_{\rm co}\}$,
and $u \in H^m_{\rm co}$ whenever
$\sum_{0 \leq |l| \leq m} \|Z_l u\|_{L^2(\Omega)} < \infty$
with $\{Z_l\}$ spanning the space of vector fields tangential to $\p\Omega$.

In this paper, by performing the higher-order energy estimates for the weak solutions constructed in \cite{cq2},
we first establish the existence and uniqueness of the strong solution of the Navier-Stokes equations
in $C(0, T_\star; H^r(\Omega; \R^3)) \cap C^1(0, T; H^{r-2}(\Omega; \R^3))$ for some $T_\star>0$ and $r\geq 2$,
subject to the kinematic and Navier boundary conditions.
We assume that domain $\Omega$ is regular, with the smooth second fundamental form $\two$.
In fact, in the estimates, we need $\|\two\|_{C^{r-1}(\po)} < \infty$.
Moreover, an explicit lower bound for $T_\star$ is obtained.
This is achieved by  employing more delicate energy estimates,
which take into account the effects of the curvature (equivalently, the second fundamental form $\two$)
of $\po$ and the Navier boundary conditions.
In addition, we study the inviscid limit (also known as the vanishing viscosity limit)
of the Navier-Stokes equations \eqref{NS equation}: We send $\nu \rightarrow 0^{+}$
 and investigate whether the strong solutions $u^\nu$ converge, in suitable norms,
 to the corresponding solution of the Euler equations describing
 the motion of incompressible, inviscid fluids:
\begin{equation}\label{Euler}
\begin{cases}
\p_t u + (u \cdot \na) u + \na p =0\qquad \text{ in }[0,T]\times\Omega,\\
\na \cdot u =0\qquad\,\,\, \text{ in }[0,T]\times\Omega,\\
u|_{t=0} = u_0 \qquad \text{ on }\Omega,
\end{cases}
\end{equation}
subject to
the {\em no-penetration boundary condition}:
\begin{equation}\label{boundary condition for Euler}
u \cdot \n = 0 \qquad \text{ on } [0,T]\times \po.
\end{equation}
	
As discussed above, for the kinematic and Navier boundary conditions, the inviscid limit problem
was answered in the affirmative for strong solutions on domains with flat boundaries ({\it e.g.}
the half-space) by Xiao-Xin \cite{xin1} and Beir\~{a}o da Veiga-Crispo \cite{v1,v2}.
This is achieved by analyzing the aforementioned simplified boundary condition in \cite{bardos, solonnikov},
which agrees with the Navier boundary condition for flat boundaries.
Similar affirmative results are also established for several modified versions of the slip-type boundary
conditions in \cite{xin2, zhong}.
In addition, the inviscid limit for the strong solutions in $L^2$ or $H^1$
under the kinematic and Navier boundary conditions are proved by Chen-Qian \cite{cq2},
Intimie-Planas \cite{iftimie1}, and Neustupa-Penel \cite{neustupa, neustupa2} for bounded,
regular, possibly non-flat domains in $\R^3$.
	
On the other hand,  recently in \cite{v3, v4}, Beir\~{a}o da Veiga-Crispo proved that
the inviscid limits in strong topologies of $W^{s,p}$ for $s >1$ and $p >1$ fails for general non-flat domains,
with the Navier-Stokes equations equipped with the simplified boundary conditions
as in \cite{bardos, solonnikov}.
In comparison, the inviscid limit in strong topologies always
holds for regular domains in 2-D, when the  Navier boundary condition
is assumed. This is largely due to the fact that the vorticity
is transported in 2-D; {\it cf.} \cite{clopeau, lopes, c3}.
	
In view of the discussions above, it is important
to understand whether the inviscid limit holds for strong solutions
in the higher-order Sobolev norms in $H^r(\Omega;\R^3)$ for $r >1$ in a bounded, regular,
generally non-flat domain $\Omega \subset \R^3$,
when the Navier-Stokes equations \eqref{NS equation} are equipped with the Navier boundary
conditions ({\it i.e.} Eq. \eqref{boundary conditions}). To the best of our knowledge,
this problem is still largely open. In Theorem \ref{theorem: inviscid limit}, we answer this question
in the affirmative: If the strong solution exists in $H^{r+1}(\Omega; \R^3)$ for $r >\frac{5}{2}$,
we establish its strong convergence in $H^r (\Omega;\R^3)$ as the viscosity constant $\nu\to 0$.
This implies
that the boundary layers do not develop up to the highest order Sobolev norm in $H^{r}(\Omega; \R^3)$ for $r>\frac{5}{2}$.

The rest of the paper is organized as follows: In \S 2
we briefly sketch the derivation of the boundary conditions in terms of geometric quantities.
In \S 3, we prove a lemma which expresses the $H^r$--norm of a divergence-free vector field
by the $L^2$--norm of the iterated curls, subject to the kinematic and Navier boundary conditions.
Next, in \S 4, we derive the {\em a priori}, higher-order energy estimates in $H^r(\Omega; \R^3)$ for $r \geq 2$
for the Navier-Stokes equations with kinematic and Navier boundary conditions.
We also deduce the existence of strong solutions from the energy estimates.
Then, in \S 5, the inviscid limit is established.
Finally, in \S 6, we discuss the inviscid limit problem for other non-Navier slip-type boundary conditions.

Before concluding this introduction, we present some notations that will be used from now on  in this paper.
We denote $H^r(\Omega; \R^3)=W^{2,r}(\Omega; \R^3)$ as the Sobolev space of vector
fields $\phi: \Omega \rightarrow \R^3$ with the norm in the multi-index notation:
\begin{equation}
\|\phi\|_{H^r(\Omega)}:=\Big(\sum_{0 \leq |\alpha| \leq r} \int_{\Omega}|\na^\alpha \phi|^2\dx\Big)^{1/2} < \infty.
\end{equation}
We write $\na^{[s]}$ to  denote a generic differential
operator $\na_{i_1}\na_{i_2}\cdots\na_{i_{s}}$ for any $s \geq 1$.
The Einstein summation convention is used. For the indices,
we write $i_1, i_2, \ldots ,j,k,l,\ldots \in \{1,2,3\}$ and $\alpha,\beta,\gamma, \delta, \ldots \in \{1,2\}$.
The angular bracket $\langle\cdot,\cdot\rangle$ denotes the Euclidean inner product of two vectors in $\R^3$.
Furthermore, we write $f \lesssim g$ if $|f| \leq C|g|$ for a generic constant $C$ depends only
on $r$, $\|\two\|_{C^{r-1}(\po)}$, and $\zeta$; and write $f \simeq g$ whenever $f \lesssim g$ and $g \lesssim f$.
Denote $\mathcal{H}^2$ as the 2-D Hausdorff measure.
Finally, $\curl^r := \curl \circ\ldots\circ\curl$ means the composition of $r$ curls.

\section{The Navier Boundary Condition}

In this section, we briefly sketch the derivation of the boundary conditions in terms of geometric quantities.

First of all, we justify that our geometric formulation of the Navier boundary
condition ({\it i.e.} the second equation in \eqref{boundary conditions}, reproduced below):
\begin{equation*}
\omega^\nu \cdot \ttau
= - \frac{1}{\zeta} (\rot u^\nu)\cdot \ttau + 2 \rot\big(\shape(u^\nu)\big) \cdot \ttau
\qquad \text{ on } \p\Omega \times [0,T] \text{ for any } \ttau \in T (\p\Omega)
\end{equation*}
is indeed equivalent to the one proposed by Navier in \cite{navier}.
For simplicity, we drop superscript $\nu$ in this section.

We start by remarking on the geometric notations.
Recall that the boundary of the domain of fluid, $\po$, is a regular surface embedded in $\R^3$.
We denote its second fundamental form by $\two: T(\po) \times T(\po) \rightarrow \R$,
where $T(\po)$ is the tangent bundle of $\po$.
Thus, writing $\n \in T(\po)^\perp$ as the outward unit
normal (viewed as the Gauss map $\n: \po \rightarrow \mathbb{S}^2$), we have
\begin{equation}
\two = -\na \n.
\end{equation}
In addition, take $\{e_1, e_2, e_3\}$ to be an orthonormal frame such
that $e_1, e_2 \in T(\po)$ and $e_3 = \n$.
Then we have the local expression:
\begin{equation}
\two(u,v) = \sum_{\alpha=1}^2\sum_{\beta=1}^2 \two_{\alpha\beta}u^\alpha v^\beta.
\end{equation}
The shape operator $\shape: T(\po) \rightarrow T(\po)$ is then defined as
\begin{equation}
\shape(u):= -\na_u \n,
\end{equation}
where $\na_u$ means the directional derivative in the direction of $u$.

Now, recall that the Navier boundary condition reads that, for any $\ttau \in T(\p\Omega)$,
\begin{equation*}
u \cdot \ttau = -2\zeta \,  \mathbb{D}u (\ttau, \n) \qquad \text{ on }  \p\Omega\times[0,T],
\end{equation*}
where, in local coordinates, the rate-of-strain tensor is given by
\begin{equation*}
(\mathbb{D}u)_{ij} = \frac{1}{2}\big(\na_i u^j + \na_j u^i\big), \qquad 1\leq i,j \leq 3.
\end{equation*}
Suppose that $\{{e}_1, {e}_2, e_3\}$ is an orthonormal moving frame adapted to $\p\Omega$, with ${e}_3=\n$.
Then the Navier boundary condition is equivalent to the following:
\begin{equation}
u^1 = -\zeta (\na_3 u^1 + \na_1 u^3),
\quad
u^2 = - \zeta (\na_3 u^2+ \na_2 u^3) \,\,\qquad \text{ on }  \p\Omega\times[0,T].
\end{equation}

The main issue of this paper is to derive the higher-order energy estimates of velocity $u$.
As shown in \S 3 below, the $H^r$--norm of $u$ is estimated purely by the $L^2$--norm
of the $r$-th iterated {\em curls} of $u$ ({\it cf.} Theorem \ref{theorem: div-curl estimate}).
We now seek for the boundary condition with respect to the vorticity: $\omega = \na \times u$.
For this purpose, note that
\begin{equation*}
\omega = \begin{bmatrix}
\na_2 u^3 - \na_3 u^2\\
\na_3 u^1 - \na_1 u^3\\
\na_1 u^2 - \na_2 u^1
\end{bmatrix}
\end{equation*}
in the local frame $\{{e}_1, {e}_2, {e}_3\}$.
Then the Navier boundary condition \eqref{classical Navier boundary condition} becomes
\begin{equation}\label{classical navier condition, local}
\na_k u^3 + \na_3 u^k = -\frac{1}{\zeta} u^k \qquad \text{ for } k \in \{1,2\}.
\end{equation}
On the other hand, $\na_k u^3$ can be computed as
\begin{equation}\label{partial k e3}
\na_k u^3 = \na_k (u \cdot \n) = \p_3 (u \cdot \n) + \sum_{j=1}^3 \Gamma^3_{kj}u^j,
\end{equation}
where $\Gamma^k_{ij} = \frac{1}{2} g^{kl} (\p_i g_{jl} + \p_j g_{il} - \p_{l} g_{ij})$ are the Christoffel symbols.
Observe also that
\begin{align}\label{two jk}
\two_{jk} = \two({e}_j, {e}_k) \cdot \n = -  \na_j {e}_k \cdot \n
= - \sum_{l=1}^3 \Gamma^l_{jk} {e}_l \cdot \n = - \Gamma^3_{jk}.
\end{align}
Then, by collecting Eqs. \eqref{classical navier condition, local}--\eqref{two jk}, we have
\begin{equation}
\na_k u^3 - \na_3 u^k = 2 \p_k (u \cdot \n) - 2 \sum_{j=1}^3 \two_{jk} u^j + \frac{1}{\zeta} u^k.
\end{equation}
Finally, in view of the kinematic boundary condition ({\it i.e.} the first equation in \eqref{boundary conditions}),
$u \cdot \n = 0$ on $\p\Omega$. Then, by taking $k=1,2$, respectively, and recalling the definition of $\rot$,
we immediately recover the second equation in \eqref{boundary conditions}.
Note that the term, $2\rot (\shape (u^\nu))$, reflects the geometry of the curvilinear fluid domain.
It vanishes when the domain is flat, {\it e.g.} the half plane.
In the rest of the paper, this is referred to as {\em the Navier boundary condition}.

\section{A Div-Curl Estimate for Divergence-free Vector Fields}

In this section, we show that the $H^{r+1}$--norm of a divergence-free vector field is equivalent to
the sum of the $L^2$--norms of its iterated curls up to the $(r+1)$-th order.
It is a variant of the well-known div-curl estimate due to Cald\'{e}ron-Zygmund
for divergence-free vector fields.

\begin{theorem}\label{theorem: div-curl estimate}
Let $u \in H^{r+1}(\Omega; \R^3)\cap K_2(\Omega)$ for $r \geq 0$ satisfy the kinematic and Navier boundary conditions \eqref{boundary conditions},
where
\begin{equation}\label{3.2a}
K_2(\Omega):= \big\{ u \in L^2(\Omega; \R^3)\,:\, \nabla\cdot u =0\big\}.
\end{equation}
Then there exists a universal constant $M=M(r, \Omega)>0$ such that
\begin{equation}
\|\na^{r+1} u \|^2_{L^2(\Omega)}  \leq M  \sum_{l=0}^{r+1} \|\curl^l \, u\|^2_{L^2(\Omega)}.
\end{equation}
\end{theorem}

Here and in the sequel, the time variable $t$ is always suppressed when only the spatial regularities are considered.
The following {\em Sobolev trace theorem} is also frequently used:

\begin{lemma}[Theorem 5.36 in \cite{sobolev}]\label{lemma: trace theorem}
Let $\Omega$ be a domain in $\R^n$ satisfying the uniform $C^m$--regularity condition.
Assume that there exists a $(m,p)$--extension operator for $\Omega$. Suppose that
\begin{equation}
mp < n, \qquad p \leq q \leq p^\ast := \frac{(n-1)p}{n-mp}.
\end{equation}
Then the continuous embedding $W^{m,p}(\Omega) \emb L^q(\po)$ holds.
\end{lemma}

In particular, it implies that $H^1(\Omega) \emb L^q(\po)$ for any $q \in [2,4]$ in the regular domain $\Omega \subset \R^3$.
\begin{proof}[Proof of Theorem {\rm \ref{theorem: div-curl estimate}}]
We prove the theorem by induction on $r$. The arguments are divided into seven steps.
	
{\bf 1.} We first establish the base case $r=0$. Indeed, in view of the following
identity (see Eq. (3.3) in Chen-Qian \cite{cq2}):
\begin{equation*}
\|\na u\|^2_{L^2(\Omega)} = \|\na\times u\|^2_{L^2(\Omega)} + \|\na \cdot u \|^2_{L^2(\Omega)}
  - \int_\po (\na \cdot u) \langle u, \n \rangle \ddh + \int_\po \langle u\cdot \na u, \n \rangle \ddh,
\end{equation*}
for the incompressible velocity field satisfying the kinematic boundary condition,  we have
\begin{equation}\label{nabla u in L2}
\|\na u\|^2_{L^2(\Omega)} = \|\na\times u\|^2_{L^2(\Omega)} + \int_\po \two(u,u)\ddh,
\end{equation}
where we have utilized the definition of the second fundamental form $\two:=-\na \n$.
Since $\|\two\|_{L^\infty(\po)} <\infty$, we bound
\begin{equation}
\big|\int_\po \two(u,u)\ddh\big| \leq \|\two\|_{L^\infty(\po)} \|u\|^2_{L^2(\po)} \leq \e \|\na u\|^2_{L^2(\Omega)} + \frac{C}{\e} \|u\|^2_{L^2(\Omega)},
\end{equation}
thanks to the Sobolev trace inequality and Young's inequality. 
Thus, the case for $r=0$ follows immediately by choosing $\e$ suitably small.

{\bf 2.} We now assume the result for $r \geq 0$ and prove it for $r+1$.
First of all, we apply integration by parts twice to obtain
\begin{align}
\|\na^{r+1} u \|^2_{L^2(\Omega)}
&=\int_{\Omega} {\big(\p_{i_1}\cdots\p_{i_{r+1}}u^k\big)\big(\p_{i_1}\cdots\p_{i_{r+1}}u^k\big)}\dx \nonumber\\
& = \int_{\Omega} \p_{i_1}\Big\{\big(\p_{i_2}\cdots\p_{i_{r+1}}u^k\big)\big(\p_{i_1}\cdots\p_{i_{r+1}}u^k\big)\Big\}\dx \nonumber\\
&\quad - \int_{\Omega} \p_{i_2} \Big\{\big(\p_{i_2}\cdots\p_{i_{r+1}}u^k\big)\big(\Delta\p_{i_3}\cdots\p_{i_{r+1}}u^k\big)\Big\} \dx \nonumber\\
&\quad + \int_\Omega \big(\Delta \p_{i_3} \cdots \p_{i_{r+1}}u^k \big)\big(\Delta \p_{i_3} \cdots \p_{i_{r+1}}u^k \big) \Big\}\dx \nonumber\\
&=: I+J+K.
\end{align}
Using the divergence theorem, the above three integrals are expressed as
\begin{equation}\label{eq: IJK}
\begin{cases}
I = \frac{1}{2}\int_\po \p_{\n}|\na^{r}u|^2 \ddh, \\
J = \int_\po \big(\p_{i_2}\cdots\p_{i_{r+1}}u^k\big)\big(\Delta\p_{i_3}\cdots\p_{i_{r+1}}u^k\big) \langle \na_{i_2}, \n \rangle \ddh,\\
K = \int_{\Omega}|\na^{r-1}\psi|^2\dx,
\end{cases}
\end{equation}	
where $\psi = \curl \, \omega = - \Delta u$ is the stream function.

{\bf 3.}
Now we bound the surface integral $I$ in \eqref{eq: IJK}.
For this purpose, we introduce a local moving frame $\{e_1, e_2, e_3\}$ on surface $\po$
such that $e_1, e_2 \in T(\po)$ and $e_3 = \n$.  Then
\begin{align}\label{I=I1+I2}
I&= \int_\po \big(\na_{i_1}\cdots \na_{i_r} u^k\big)\big(\na_{i_1}\cdots \na_{i_r}\na_3 u^k\big) \ddh
+ \int_\po \big(\na_{i_1}\cdots \na_{i_r} u^k\big) \big([\na_3, \na_{i_1}\cdots \na_{i_r}] u^k \big) \ddh \nonumber\\
&=: I^1 + I^2,
\end{align}
where $[\cdot,\cdot]$ denotes the commutator.
Since the commutator is of lower order, the second term in the integrand of $I^2$ is schematically represented
as $\na^{[r-1]}u^k$. More precisely, by the Ricci identity:
\begin{equation}
\na_i \na_j V^k - \na_j \na_i V^k = \sum_{l} C_{ij}^{kl} V_l
\end{equation}
for any vector field $V \in T\R^3$ and some constants $C^{kl}_{ij}$, each time we exchange $\na_3$ with $\na_{i_j}$, a zero-th order term is obtained.
Then the Leibniz rule yields
\begin{equation}
[\na_3, \na_{i_1}\cdots \na_{i_r}] u^k  \simeq \na^{[r-1]}u^k.
\end{equation}
Then the Cauchy-Schwarz inequality leads to
\begin{align}\label{I2}
|I^2| \lesssim \|u\|_{H^r(\po)}^2 + \|u\|_{H^{r-1}(\po)}^2
\lesssim \e\| \na^{r+1}u \|^2_{L^2(\Omega)} + (1+\frac{1}{\e}) \|u\|_{H^r(\Omega)}^2,
\end{align}
where the second line follows from the Sobolev trace embedding $H^{r+1}(\Omega) \emb H^r(\po)$
for $r \geq 0$, together with the interpolation inequalities.
	
{\bf 4.} To bound $I^1$, we make a crucial use of the kinematic and Navier boundary conditions \eqref{boundary conditions}.
First, we rewrite it in the local frame $\{e_1, e_2, e_3\}$ as
	\begin{equation}\label{BC in local frame}
	\begin{cases}
	u^3 = 0, \\
	\na_3 u^\beta = 2 \two_{\alpha \beta}u^\alpha - \frac{1}{\zeta} u^\beta \qquad \text{ for } \beta \in \{1,2\},
	\end{cases}
	\end{equation}
where $\na_\alpha u^3 \equiv 0$ so that $\omega^1 = -\na_3 u^2$ and $\omega^2=\na_1 u^3$.
Moreover, from the incompressibility condition: $\na \cdot u =0$,
the following identities hold:
\begin{equation}\label{identities for na 3}
\begin{cases}
\na_3 u^3 = - \na_\alpha u^\alpha,\\
\na_3\na_3 u^\alpha = - \psi^\alpha - \na_\beta \na_\beta u^\alpha.
\end{cases}
\end{equation}
The key to  Eqs. \eqref{BC in local frame}--\eqref{identities for na 3} is that
the {\em normal derivatives} $\na_3$ of the normal components can be replaced
by the tangential derivatives, and the normal derivatives of the tangential components
can be replaced by the lower-order terms.

{\bf 5.}
We now estimate $I^1$.
For simplicity, we introduce the short-hand notations:
\begin{equation}
\nr A \cdot \nr B :=  \big(\na_{i_1}\cdots\na_{i_{r-3}}A\big) \cdot \big(\na_{i_1}\cdots\na_{i_{r-3}}B\big),
\end{equation}
for any  sufficiently regular functions $A$ and $B$.
Then we split $I^1$ into six terms:
\begin{equation*}
I^1:= I^{1,1} + I^{1,2} +I^{1,3} + I^{1,4} + I^{1,5} + I^{1,6},
\end{equation*}
where
\begin{equation}
\begin{cases}
I^{1,1} = \int_\po\,(\nr \na_\alpha \na_\beta u^3) \cdot (\nr \na_\alpha\na_\beta \na_3 u^3)\ddh,\\
I^{1,2} = \int_\po\, (\nr \na_\alpha \na_\beta u^\gamma) \cdot (\nr \na_\alpha\na_\beta \na_3 u^\gamma)\ddh,\\
I^{1,3} = \int_\po\, (\nr \na_\alpha \na_3 u^3) \cdot (\nr \na_\alpha\na_3 \na_3 u^3)\ddh,\\
I^{1,4} = \int_\po\, (\nr \na_\alpha \na_3 u^\gamma) \cdot (\nr \na_\alpha\na_3 \na_3 u^\gamma)\ddh,\\
I^{1,5} = \int_\po\, (\nr \na_3 \na_3 u^3) \cdot (\nr \na_3\na_3 \na_3 u^3)\ddh,\\
I^{1,6} = \int_\po\, (\nr \na_3 \na_3 u^\gamma) \cdot (\nr \na_3\na_3\na_3 u^\gamma)\ddh.
\end{cases}
\end{equation}
In the sequel, we estimate these terms one by one.

First of all, $I^{1,1} = 0$, since $\na_\beta u^3\equiv 0$.

To estimate $I^{2,2}$, we first notice that
\begin{align}\label{x}
\big|\nr \na_\alpha\na_\beta \na_3 u^\gamma \big|
=\big| \nr \na_\alpha\na_\beta\big(2\two_{\gamma\beta}u^\beta - \frac{1}{\zeta}u^\gamma\big)\big|
= C |\na^{[r-1]} u| + {\rm {l.o.t.}},
\end{align}
where $C$ depends on $\|\two\|_{C^{r-1}(\po)}$ and $\zeta^{-1}$, and ${\rm l.o.t.}$ contains the derivatives of $u$ of order
less than or equal to $r-2$.
Next, considering the two cases: $\alpha = \beta$ and $\alpha \neq \beta$ separately, we deduce
\begin{align}
I^{1,2} =& \int_\po \big(\nr \Delta u^\gamma\big)\cdot \big(\nr \Delta \na_3 u^\gamma\big) \ddh \nonumber\\
 &+ 2 \int_{\po}\big(\nr \na_1\na_2 u^\gamma\big)\cdot\big(\nr \na_1\na_2\na_3 u^\gamma\big) \ddh.
\end{align}
For the first term, again by the Ricci identity, we write
\begin{equation}
\na^{[r-3]} \Delta \na_3 u^\gamma = \nr \na_3 \Delta u^\gamma + \na^{[r-3]}(\Delta u^\gamma) = - \na^{[r-3]} \na_3 \psi^\gamma - \na^{[r-1]}u^\gamma,
\end{equation}
and treat the second term as in Eq. \eqref{x} above. Then we obtain
\begin{align}
|I^{1,2}| &\lesssim \Big|\int_\po \big(\na^{[r-3]} \na_3 \psi^\gamma\big)\cdot \big( \na^{[r-1]}u^\gamma\big)\ddh\Big| 
+ \int_\po |\na^{[r-1]}u|^2\ddh \nonumber\\ &\lesssim  \|u\|^2_{H^{r-1}(\po)} + \|\na^{[r-3]} (\na \times \psi)\|^2_{L^2(\po)}
\end{align}
by the Cauchy-Schwarz inequality.
By the trace and interpolation inequalities, we have
\begin{equation*}
\|\na^{[r-3]} (\na \times \psi)\|^2_{L^2(\po)} \lesssim \e \|\na^{r+1}u\|^2_{H^{r+1}(\Omega)} + \frac{1}{\e} \|u\|^2_{H^r(\Omega)}.
\end{equation*}
Then
\begin{equation}\label{I1,2}
|I^{1,2}| \lesssim \e \|\na^{r+1}u\|^2_{H^{r+1}(\Omega)} + \frac{1}{\e} \|u\|^2_{H^r(\Omega)}.
\end{equation}

For $I^{1,3}$, again by Eq. \eqref{identities for na 3}, the Ricci identity, the boundary condition \eqref{BC in local frame},
and the trace and interpolation inequalities, we have
\begin{align}\label{I1,3}
|I^{1,3}|  &= \Big|\int_\po \big(\nr \na_\alpha \na_\beta u^\beta\big)\big( \nr \na_\alpha\na_3 \na_\gamma u^\gamma\big) \ddh\Big| \nonumber\\
&\simeq \Big|\int_\po \big(\nr \na_\alpha \na_\beta u^\beta\big)\big(\nr \na_\alpha \na_\gamma
\na_3 u^\gamma + \na^{[r-1]}u\big)\ddh\Big| \nonumber\\
&\simeq  \Big|\int_\po \Big(\nr \na_\alpha \na_\beta u^\beta\Big)\Big(\nr \na_\alpha \na_\gamma
\big(2\two_{\delta\gamma} u^\delta- \frac{1}{\zeta} u^\gamma\big) + \na^{[r-1]}u\Big)\ddh\Big| \nonumber\\
&\lesssim \|u\|^2_{H^{r-1}(\po)} \lesssim \|u\|^2_{H^r(\Omega)}.
\end{align}

The treatment for $I^{1,4}$ is similar to the above for $I^{1,3}$:
\begin{align}\label{I1,4}
|I^{1,4}| &= \Big|\int_\po \Big(\nr \na_\alpha \big(2\two_{\beta\gamma}u^\beta
  - \frac{1}{\zeta}u^\gamma\big)\Big)\Big(\nr \na_\alpha \big(-\psi^\alpha - \na_\delta\na_\delta u^\gamma \big)\Big) \ddh\Big| \nonumber\\
&\lesssim \Big|\int_\po \big(\na^{[r-2]}u\big)\big(\na^{[r]}u\big)\ddh\Big| \nonumber\\
&\lesssim \|u\|^2_{H^r(\po)} \lesssim \e \| \na^{r+1}u \|^2_{L^2(\Omega)} + \frac{1}{\e} \|u\|^2_{H^r(\Omega)}.
\end{align}

For $I^{1,5}$, we first substitute in Eq. \eqref{BC in local frame} to derive
\begin{equation*}
I^{1,5} = \int_\po \big(\nr \na_3\na_\beta u^\beta\big) \big(\nr\na_3\na_3 \na_\alpha u^\alpha\big) \ddh.
\end{equation*}
Then, applying the Ricci identity once to the first term and twice to the second term in the integrand, we have
\begin{align}\label{I1,5}
|I^{1,5}| &\lesssim \Big|\int_\po \big(\nr \na_\beta \na_3 u^\beta + \na^{[r-2]}u\big) \big(\nr \na_\alpha \na_3\na_3 u^\alpha + \na^{[r-1]}u\big) \ddh\Big|\nonumber\\
&\simeq \Big|\int_\po \Big(\nr \na_\beta \big(2\two_{\beta\gamma}u^\gamma -\frac{1}{\zeta} u^\gamma \big) + \na^{[r-2]}u\Big) \nonumber\\
&\qquad\qquad \times \Big(\nr \na_\alpha \big(-\psi^\alpha -\na_\delta\na_\delta u^\alpha \big) + \na^{[r-1]}u\Big) \ddh\Big| \nonumber\\
&\lesssim \Big|\int_\po \big(\na^{[r-2]}u\big)\big(\na^{[r]}u\big)\ddh\Big| \lesssim \e \| \na^{r+1}u \|^2_{L^2(\Omega)} +  \frac{1}{\e} \|u\|^2_{H^r(\Omega)},
\end{align}
where the last line follows analogously to the final  inequality in Eq. \eqref{I1,4}.

Finally, for $I^{1,6}$, using Eqs. \eqref{BC in local frame}--\eqref{identities for na 3}, we have
\begin{align}\label{I1,6}
|I^{1,6}| & =\Big|\int_\po \Big(\nr \big\{-\psi^\alpha - \na_\beta\na_\beta u^\alpha \big\}\Big)\Big(\nr \na_3\na_3\big(2\two_{\alpha\gamma} u^\gamma - \frac{1}{\zeta} u^\gamma\big)\Big)\ddh\Big| \nonumber\\
&\lesssim \|\na^{[r-1]} u\|^2_{L^2(\po)} \lesssim \|u\|^2_{H^r(\Omega)}.
\end{align}

Therefore, combining Eqs. \eqref{I1,2}--\eqref{I1,6} all together,
$I^1$ is estimated by
\begin{equation}
|I^1| \lesssim \e \| \na^{r+1}u \|^2_{L^2(\Omega)}  + \frac{1}{\e} \|u\|^2_{H^{r}(\Omega)}.
\end{equation}

{\bf 6.} Now we derive the estimates for the $J$ term.

In fact, $J$ differs from $I$ only by the lower-order terms so that the estimates follow immediately.
More precisely,  notice that
\begin{align}
J & = \int_\po \big(\na_{i_2}\nr \na_{i_{r+1}} u^k\big)\big(\Delta\nr \na_{i_{r+1}} u^k\big) \langle \na_{i_2}, \n\rangle\ddh,
\end{align}
where we have relabelled $\nr = \na_{i_3}\cdots \na_{i_{r+1}}$ as before.
Then, invoking the Ricci identity again, it follows that
\begin{align}
J &\lesssim \int_\po \big(\nr \na_{i_2}\na_{i_{r+1}}u^k + \na^{[r-2]}u\big)\big(\nr \na_{i_{r+1}}\Delta u^k + \na^{[r-1]}u\big) \ddh\nonumber\\
&= \int_\po\big(\nr \na_{i_2}\na_{i_{r+1}}u^k \big)\big(\nr \na_{i_{r+1}}\Delta u^k\big)\ddh+ \int_\po\big(\na^{[r-2]}u\big)\big(\na^{[r-1]}u\big)\ddh\nonumber\\
&\quad + \int_\po\big( \na_{i_{r+1}}\nr\Delta u^k\big)\big(\na^{[r-2]}u\big)\ddh + \int_\po\big(\na^{[r-1]}u\big)\big(\na^{[r-1]}u\big)\ddh \nonumber\\
&=: J^1 + J^2+J^3+J^4.
\end{align}
By the trace, interpolation, and Young's inequalities, again we have
\begin{equation}
|J^2|+ |J^4| \lesssim \|\na^{[r-1]} u\|_{L^2(\po)}^2 \lesssim \e \|\na^{r+1}u\|^2_{L^2(\Omega)} + \frac{1}{\e} \|u\|^2_{H^r(\Omega)}.
\end{equation}
Also, $J^1$ has the same decomposition as $I^1$ into $I^{1,1}, \ldots, I^{1,6}$ so that, by Step 4, we conclude
\begin{equation}
|J^1| \lesssim \e \|\na^{r+1}u\|^2_{L^2(\Omega)} + \frac{1}{\e} \|u\|^2_{H^r(\Omega)}.
\end{equation}
In the end, $J^3$ is estimated via integration by parts again: Since $\po$ is a 2-D surface without boundary,
the divergence theorem yields
\begin{equation}
J^3 = - \int_\po \big(\nr \Delta u^k\big) \big(\na_{i_{r+1}} \na^{[r-2]}u\big)\ddh \simeq \int_\po \big(\na^{[r-1]}u\big) \big(\na^{[r-1]}u\big)\ddh.
\end{equation}
Thus, this verifies the same estimate for $J^4$.

{\bf 7.}
Finally, putting together all the estimates for $I,J$, and $K$ in Steps 1--6, we conclude
\begin{equation}
 \|\na^{r+1} u \|^2_{L^2(\Omega)}  \lesssim \e\|\na^{r+1} u \|^2_{L^2(\Omega)} + \frac{1}{\e} \|u\|^2_{H^r(\Omega)} +  \int_{\Omega}|\na^{r-1}\psi|^2\dx.
\end{equation}
Choose $\e$ sufficiently small so that
\begin{equation}\label{r+1 derivative estimate, in terms of K}
\|\na^{r+1} u \|^2_{L^2(\Omega)}  \lesssim \|u\|^2_{H^r(\Omega)} + K,
\end{equation}
where $K:= {\int_\Omega |\na^{r-1}\psi|^2 \dx}$ as before.
The first term on the right-hand side, $\|u\|^2_{H^r(\Omega)}$, is bounded
by $\sum_{l=0}^r\|\curl^l \, u\|^2_{L^2(\Omega)}$ up to a multiplicative constant,
thanks to the induction hypothesis.

Now, it remains to show that $K$ is bounded by the $L^2$--norm of the iterated curls:
This is achieved by iterating the constructions in Step $1$.
Indeed, relabelling the indices yields
\begin{equation*}
K = \int_\Omega \big(\Delta \p_{i_1} \cdots \p_{i_{r-1}}u^k\big)\big(\Delta \p_{i_1} \cdots \p_{i_{r-1}}u^k\big) \dx.
\end{equation*}
Then, as in Step 1, we integrate by parts twice to compute as
\begin{align}
K=&\,  \int_{\po} \big(\Delta\p_{i_2}\cdots\p_{i_{r-1}}u^k\big) \big(\Delta\p_{i_1}\cdots \p_{i_{r-1}}u^k\big)\langle\p_{i_1},\n\rangle \ddh\nonumber\\
&\,-\int_{\po} \big(\Delta \p_{i_2}\cdots\p_{i_{r-1}}u^k\big) \big(\Delta\Delta\p_{i_3}\cdots \p_{i_{r-1}}u^k\big)\langle\p_{i_2},\n\rangle \ddh \nonumber\\
&\, + \int_\Omega\big(\Delta\Delta \p_{i_3}\cdots\p_{i_{r-1}}u^k\big)\big(\Delta\Delta \p_{i_3}\cdots\p_{i_{r-1}}u^k\big) \dx \nonumber\\
=&\,\frac{1}{2}\int_\po \p_{\n} |\na^{r-2}\Delta u|^2\ddh
 -\int_{\po} \big(\Delta \p_{i_2}\cdots\p_{i_{r-1}}u^k\big) \big(\Delta\Delta\p_{i_3}\cdots \p_{i_{r-1}}u^k\big)\langle\p_{i_2},\n\rangle \ddh\nonumber\\
&\, + \int_\Omega |\Delta\Delta \na^{r-3}|^2 u^2 \dx =: \tilde{I} + \tilde{J} + \tilde{K}.
\end{align}
It is crucial here that the flat gradient $\p_{i_l}$ and flat Laplacian $\Delta$ on $\R^3$ commute.
We notice that
$\tilde{I}$ and $\tilde{J}$ are obtained from $I$ and $J$, respectively,
by taking the trace over a pair of indices, so that they satisfy the same estimates, which are given in Step $5$ above.
Repeating this process for finitely many times, we refine estimate \eqref{r+1 derivative estimate, in terms of K} as
\begin{equation}\label{r+1 derivative estimate, in terms of tilde-K}
\|\na^{r+1} u \|^2_{L^2(\Omega)}  \lesssim \sum_{l=0}^r \|\curl^l \, u\|^2_{L^2(\Omega)} +  \begin{cases}
\int_\Omega |\Delta^{\frac{r+1}{2}}u |^2\dx \quad \text{ if $r$ is odd},\\[1mm]
\int_\Omega |\Delta^{\frac{r}{2}} \na u|^2 \dx \quad \text{ if $r$ is even}.
\end{cases}
\end{equation}

To conclude the proof, we notice that, for the divergence-free vector field $u$,
$\Delta u = -\curl^2 u$.
Thus, Eq. \eqref{r+1 derivative estimate, in terms of tilde-K} gives the desired estimate for odd $r$.
On the other hand, for even $r$, we apply Eq. \eqref{nabla u in L2} in Step 1 of the same proof to
the divergence-free vector field $\Delta^{\frac{r}{2}}u$ to deduce
\begin{equation}
\int_\Omega |\Delta^{\frac{r}{2}} \na u|^2\dx
= \int_\Omega |\na\Delta^{\frac{r}{2}}  u|^2\dx
= \int_\Omega |\curl(\Delta^{\frac{r}{2}}u)|^2\dx
 + \int_\po \two(\Delta^{\frac{r}{2}}u,\Delta^{\frac{r}{2}}u)\ddh,
\end{equation}
where we need the commutativity of divergence, gradient, and curl.
For the first term on the right-hand side,
$\curl(\Delta^{\frac{r}{2}}u) = (-1)^{\frac{r}{2}} \curl^{r+1} u$,
while, for the second term,
\begin{equation}
\int_\po \two(\Delta^{\frac{r}{2}}u,\Delta^{\frac{r}{2}}u)\ddh \lesssim \|\Delta^{\frac{r}{2}}u\|^2_{L^2(\po)}
\simeq \|\na^r u\|^2_{L^2(\po)} \lesssim \e' \|\na^{r+1}\|^2_{L^2(\Omega)} + \frac{1}{\e'} \|u\|^2_{H^r(\Omega)},
\end{equation}
again by the boundedness of the second fundamental form,
as well as the trace and interpolation inequalities.
The proof is then completed by choosing $\e'$ sufficiently small.
\end{proof}

To conclude the section, we emphasize that Theorem \ref{theorem: div-curl estimate} is independent of
the Navier-Stokes equations \eqref{NS equation}.
It is a general property of  divergence-free vector fields satisfying the kinematic and Navier boundary
conditions \eqref{boundary conditions}.
For the Dirichlet boundary condition, Theorem 3.1 also holds, which follows from the divergence-free condition.

\section{Energy Estimates in $H^r$ for Strong Solutions}

In this section, we derive the higher-order energy estimates.
We show that the solution is in the spatial Sobolev space $H^r$ for $r \geq 2$,
provided that the initial data lies in the same space.
This allows us to prove the existence of strong solutions with spatial regularity $H^r$.

For this purpose, our starting point is the existence of {\em weak solutions} to the Navier-Stokes equations \eqref{NS equation}
under the kinematic and Navier boundary conditions.
This can be established, {\it e.g.} via the Galerkin approximation scheme in \cite{cq2}.
We summarize it here for the subsequent developments.
In this section, we drop superscript $\nu$ in solution $u^\nu$
of the Navier-Stokes equations \eqref{NS equation},
since we do not deal with the inviscid limits here.

To begin with,
consider the following vector space direct sum
\begin{equation}
L^2(\Omega; \R^3) = K_2(\Omega) \bigoplus G_2(\Omega)
\end{equation}
for the Hodge (or Helmholtz) decomposition, where $K_2(\Omega)$ is defined in \eqref{3.2a}.
Next, for projection $\proj$ onto the first factor,
we introduce the {\em Stokes operator}:
\begin{equation}
S:=\proj\circ \Delta,
\end{equation}
where $\Delta$ is the flat Laplacian on $\R^3$.
It is shown in \S 4 of \cite{cq2} that $S$ is densely defined on $K_2(\Omega)$ with a compact resolvent.
Thus, it has a discrete spectrum $\lambda_1 \geq \lambda_2 \geq \lambda_3\geq \ldots \downarrow -\infty$,
and the corresponding eigenfunctions $\{a_n\}$ form a complete orthonormal basis of $K_2(\Omega)$.
Now we look at the graded chain of finite-D Hilbert spaces:
\begin{equation}
K_2(\Omega) \supset \ldots \supset V_N := \bigoplus_{j=1}^{N} \R a_j \supset V_{N-1} \supset \ldots \supset V_2 \supset V_1,
\end{equation}
and denote by $\mathbb{P}_N: K_2(\Omega) \rightarrow V_N$ the canonical projection:
\begin{equation}
(\mathbb{P}_N u) (t,x) := \sum_{j=1}^N a_j(x) \int_\Omega \langle a_j(y), u(t,y)\rangle \, \dd y.
\end{equation}
Thus, $\proj$ is indeed the $L^2$-limit of $\mathbb{P}_N$, as $N$ tends to $\infty$.

In \S 5 of  \cite{cq2}, the {\em weak formulation} of Eq. \eqref{NS equation} has been introduced.

\begin{definition}
For $T>0$, we say that $u \in L^2([0,T]; H^1(\Omega; \R^3))$ is a weak solution of the initial boundary problem
\eqref{NS equation}--\eqref{classical Navier boundary condition},
provided that
\begin{enumerate}
\item[\rm (i)]
$u(t,\cdot) \in K_2(\Omega)$ for each $t\in (0,T)${\rm ;}
\item[\rm (ii)]
For each $\phi \in C^\infty([0,T]\times \Omega)$ with $\phi(t,\cdot) \in K_2(\Omega)$,
\begin{align}
&\int_{\Omega}\langle u(T,\cdot), \phi(T,\cdot) \rangle \dx \nonumber\\
&= \int_{\Omega}\langle u_0(x), \phi(0,x)\rangle \dx + \int_0^T \int_\Omega \langle u, \p_t \phi\rangle \dx \,\dd t
  -\int_0^T \int_\Omega \big\langle \curl \, u, (u\times \phi + \nu \, \curl\, \phi)\big\rangle \dx\,\dd t \nonumber\\
&\quad - \frac{\nu}{\zeta}\int_0^T\int_\po \langle u, \phi\rangle \ddh \dt + 2\nu \int_0^T\int_\po \two(u, \phi)\ddh \dt;
\end{align}
\item[\rm (iii)]
The energy inequality holds{\rm :}
\begin{align}\label{energy inequality for NS weak solution}
&\|u(T,\cdot)\|^2_{L^2(\Omega)}+ 2\nu \int_0^T \|\na u(t,\cdot)\|^2_{L^2(\Omega)}\dt
 + 2 \nu \int_0^T \int_{\po}\big(\frac{1}{\zeta} |u|^2 - \two(u,u)\big)\ddh\dt \nonumber\\
&\leq \|u_0\|^2_{L^2(\Omega)}.
\end{align}
\end{enumerate}
\end{definition}

Therefore, by solving the projected equations obtained via taking $\mathbb{P}_N$ to Eq. \eqref{NS equation}
and deriving the {\em a priori} estimates for the finite-D approximate solutions $\{u_N\} \subset  L^2([0,T]; H^1(\Omega; \R^3))$ uniformly in $N$,
we are able to deduce the existence of weak solutions via a compactness argument.
This method is known as the Galerkin approximation scheme, which relies crucially on the spectral analysis
of the Stokes operator $S = \proj\circ\Delta$.

More precisely,  the following result is obtained:

\begin{lemma}[Theorem 5.1 in \cite{cq2}]\label{galerkin}
For any $u_0 \in K_2(\Omega)$ and $T>0$, there exists a weak solution $u \in L^2([0,T]; K_2(\Omega))$ to
the initial-boundary problem \eqref{NS equation}--\eqref{classical Navier boundary condition}.
Such a solution $u$ can be obtained as a weak subsequential limit of the family  of finite-D approximate solutions $\{u_N\}$.
\end{lemma}

Now, taking the weak solution $u$ of the Navier-Stokes equations \eqref{NS equation} constructed by the Galerkin approximation scheme
in Lemma \ref{galerkin} above, we derive the {\em a priori} estimate for the higher-order energy of $u$ in the Sobolev spaces $H^r$
with $r \geq 2$. Indeed, the case, $r=2$, has been proved in Theorem 5.3 of \cite{cq2}.
The higher-order energy estimate is proved by induction on $r$, for which purpose the reduction of order of differentiations
in the boundary terms is essential. This is achieved by exploiting by the kinematic and Navier boundary conditions \eqref{boundary conditions}.
In particular, we need to explore the role of the curl operator, the rotation matrix $\rot$, and the shape operator $\shape$ (see \S 1).

Our main theorem of this section is the following:

\begin{theorem}\label{theorem: higher order energy estimate}
Let $u_0 \in H^r(\Omega; \R^3) \cap K_2(\Omega)$ for some $r \geq 2$.
Then there exists some $T_\star > 0$ such that the weak solution $u\in L^2([0, T_\star); K_2(\Omega))$
of the initial-boundary problem \eqref{NS equation}--\eqref{classical Navier boundary condition}
satisfies
\begin{equation}
\sup_{0\leq t\leq T_\star} \Big(\|u(t,\cdot)\|_{H^r(\Omega)} +  \|\p_t u(t,\cdot)\|_{H^{r-2}(\Omega)}\Big) \leq C,
\end{equation}
where constant $C>0$ depends only on $\zeta$, $\nu$, $\|\two\|_{C^{r-1}(\po)}$, and $\|u_0\|_{H^r(\Omega)}$.
As a consequence, there exists a unique strong solution $u \in C([0,T_\star); H^r(\Omega; \R^3)) \cap C^1([0,T_\star); H^{r-2}(\Omega;\R^3))$.
\end{theorem}

\begin{proof}
We divide the arguments in six steps.
In Step 1, we set up the equations for the energy estimate. Then, in Steps 2--5,
we control $\|u(t,\cdot)\|_{H^r(\Omega)}$ and specify the lifespan, $T_\star$.
Finally, in Step 6, we derive the energy estimate for $\p_t u$.

{\bf 1.}
We first deduce the evolution equation for the iterated curls of the velocity field $u$.
For this purpose, we apply the divergence-free projection $\proj: L^2(\Omega; \R^3) \rightarrow K_2(\Omega)$
to the Navier-Stokes equations \eqref{NS equation} to obtain
\begin{equation}\label{projected Navier-Stokes}
\p_t u - \nu \Delta u + \proj (u \cdot \na u) = 0.
\end{equation}
On the other hand, we have the following vectorial identity in 3-D:
\begin{equation*}
u \cdot \na u = \frac{1}{2} \na (|u|^2) - u \times \omega,
\end{equation*}
so that the projected Navier-Stokes equations \eqref{projected Navier-Stokes} are equivalent to
\begin{equation}\label{eq: u}
\p_t u - \nu \Delta u + \proj (u \times \omega) = 0.
\end{equation}
Here and in the sequel, we view $\mathbb{P}_\infty$ as extended to the bounded projection
operator from $H^r(\Omega; \R^3)$ to $ H^r(\Omega; \R^3)\cap K_2(\Omega)$.
This follows from the generalized Hodge decomposition theory on the manifolds with boundaries
subject to the kinematic boundary condition; see Theorem 2.4.2
in Schwarz \cite{Schwarz}.
Then the Stokes' operator:
$$
S := \mathbb{P}_\infty \circ \Delta
$$
gives rise to a densely defined, closable, self-adjoint bilinear form on $H^r(\Omega;\R^3)\cap K_2(\Omega)$:
\begin{equation}
E_r(u,w) := -\int_\Omega \sum_{0 \leq |\alpha| \leq r} \big\langle \nabla^\alpha Su, \nabla^\alpha w \big\rangle \,{\rm d}x.
\end{equation}
In particular, the spectral analysis in Sections 4.1--4.2 in \cite{cq2}
also carries through in our setting to $H^r(\Omega;\R^3)$.

For simplicity of presentation, we use the following abbreviation:
\begin{equation}\label{Psi}
\Psi:= \proj (u\times \omega).
\end{equation}
Then, taking the iterated curls to Eq. \eqref{eq: u}, we obtain the evolution equation:
\begin{equation}\label{eq: evolution of q_r}
\p_t \q_r - \nu \Delta\q_r + \curl^r \, \Psi = 0,
\end{equation}
where and in the sequel, we denote
\begin{equation}
\q_r := \curl^r \, u.
\end{equation}

To derive the energy estimate, we multiply $\q_r$ to Eq. \eqref{eq: evolution of q_r} and integrate over $\Omega$ to obtain
\begin{equation}\label{eq:x}
0=\frac{1}{2} \frac{d}{dt}\int_\Omega |\q_r|^2 \dx - \nu \int_\Omega \langle \q_r, \Delta \q_r \rangle \dx + \int_\Omega \langle \q_r, \curl^r \, \Psi \rangle \dx.
\end{equation}
We integrate the last two terms by parts. For the second term, we have
\begin{equation*}
\int_\Omega \langle \q_r, \Delta \q_r \rangle \dx = \int_\po \langle (\q_r \cdot \na) \q_r, \n\rangle \ddh - \int_\Omega|\na \q_r|^2 \dx.
\end{equation*}
For the final term, notice that, for any $3$-D vector fields $V$ and $W$,
\begin{align}\label{integration by parts, for curl}
\int_\Omega \langle V, \curl\, W \rangle \dx & = \int_\Omega V^k \e^{ijk}\p_i W^j \dx \nonumber  \\
&= \int_\po \e^{ijk} V^k W^j \langle \p_i, \n\rangle \ddh - \int_\Omega \e^{ijk} W^j (\p_i V^k) \dx\nonumber \\
&= \int_\po \langle W \times V, \n \rangle \ddh + \int_\Omega \langle \curl\, V, W \rangle \dx.
\end{align}
As a result,
\begin{equation*}
\int_\Omega \langle \q_r, \curl \circ \curl^{r-1} \Psi \rangle \dx = \int_{\po} \langle\curl^{r-1}\, \Psi \times \q_r, \n \rangle \ddh + \int_\Omega \langle \curl\, \q_r , \curl^{r-1}\,\Psi  \rangle \dx,
\end{equation*}
so that Eq. \eqref{eq:x} can be written as
\begin{align}\label{eq: higher order energy, basic equation}
&\frac{1}{2}\frac{d}{dt} \int_\Omega |\q_r|^2 \dx + \nu \int_\Omega |\na \q_r|^2 \dx  \nonumber\\
&= \nu  \int_\po \langle (\q_r \cdot \na) \q_r, \n\rangle \ddh - \int_{\po} \langle\curl^{r-1}\, \Psi \times \q_r, \n \rangle \ddh - \int_\Omega \langle \curl\, \q_r , \curl^{r-1}\,\Psi  \rangle \dx \nonumber\\
&=I+J+K.
\end{align}

Our task is to estimate each of terms $I,J$, and $K$.
Since the case, $r=2$, has been established in Theorem 5.3 of \cite{cq2},
in the sequel, we assume the result for $r-1$ and prove it for $r$ by induction, with $r\geq 3$.

{\bf 2.} For $I$ in Eq. \eqref{eq: higher order energy, basic equation}, observe that
\begin{equation}
I = \frac{\nu}{2} \int_\po \p_\n |\q_r|^2 \ddh,
\end{equation}
which has been treated in the proof of Theorem \ref{theorem: div-curl estimate}.
Indeed, it coincides with $I$ in Eq. \eqref{eq: IJK} up to a constant $\nu$.
Utilizing the estimates in Steps 2--4 of the proof therein, we have
\begin{equation}\label{I for higher order energy}
|I| \lesssim \e \|\na \q_r\|_{L^2(\Omega)}^2 + \frac{1}{\e}\|u\|^2_{H^r(\Omega)}.
\end{equation}

{\bf 3.} To prove for term $K$ in Eq. \eqref{eq: higher order energy, basic equation},
we first notice that, by the Cauchy-Schwarz inequality and Young's inequality,
\begin{align}\label{K for higher order energy}
|K| & \leq \|\curl \, \q_r\|_{L^2(\Omega)} \|\curl^{r-1} \, \Psi\|_{L^2(\Omega)} \nonumber\\
&\lesssim \e \|\na \q_r\|^2_{L^2(\Omega)} + \frac{1}{\e} \|\curl^{r-1}\,(\proj (u \times \omega))\|^2_{L^2(\Omega)}.
\end{align}
Since $H^s(\Omega)$ for $s > \frac{3}{2}$ is a Banach algebra on $\R^3$ and $\proj$ is a bounded linear operator,
we have
\begin{equation}\label{curl Psi}
 \|\curl^{r-1} \, (\proj (u \times \omega))\|_{L^2(\Omega)} \lesssim \|u\|^2_{H^r(\Omega)}.
\end{equation}
This gives us the estimate for $K$.

{\bf 4.} Now it remains to control the boundary term $J$ in Eq. \eqref{eq: higher order energy, basic equation}:
\begin{equation}
J:= \int_\po \langle \curl^{r-1} (\proj (u \times \omega)) \times \q_r, \n \rangle \ddh.
\end{equation}
It is crucial to reduce the order of differentiation by using the boundary conditions \eqref{boundary conditions}.
For this purpose, we establish the following {\em identity} on $\partial\Omega$:
\begin{align}\label{identity for boundary, order reduction}
&\pi\big\{\curl^k (\proj(u\times \omega))\big\} \nonumber\\
&=-\frac{1}{\zeta}  \rot \circ \pi \big\{\curl^{k-1}(\proj(u\times \omega))\big\}
+ 2\rot\circ \pi \big\{\curl^{k-1}\,\shape (\pi \circ \proj(u \times \omega))\big\}.
\end{align}
We recall that $\rot$ is the orthogonal matrix rotating in the $(x,y)$--plane anti-clockwise by $90$ degrees,
$\shape$ is the shape operator corresponding to the second fundamental form $\two$,
and operator $\pi$ denotes the projection onto the tangential components of a vector field,
\begin{equation}
\pi(V) := V-\langle V, \n \rangle
\end{equation}
viewed either as a 2-D vector or a 3-D vector with zero $x_3$--component.
Here it suffices to consider the tangential components, since $\langle \n \times \q_r, \n \rangle \equiv 0$.

The above {\em identity} is proved by induction. The base step $k=1$ is shown in the computations preceding Eq. (5.21) in \cite{cq2}.
Now we assume the result for $k$. Then, by the induction hypothesis,
\begin{align}\label{boundary curl 1}
&\pi\big\{\curl^{k+1} \,(\proj(u\times \omega))\big\} \nonumber\\
&= \pi\big\{\curl\circ \curl^k  \,(\proj(u\times \omega))\big\} \nonumber\\
&= \pi\circ \curl \big\{-\frac{1}{\zeta} \rot\circ\pi\circ \curl^{k-1} \,(\proj(u \times \omega))
   + 2 \rot\circ\pi\circ \curl^{k-1}\,\shape \, (\pi\circ\proj (u\times \omega)) \big\}\nonumber\\
&=: \pi\circ \curl \big\{ -\frac{1}{\zeta} \rot\circ\pi\circ \curl^{k-1} \,\Psi
  + 2\rot\circ\pi\circ\curl^{k-1} ( \shape \circ \pi(\Psi)) \big\},
\end{align}
where we recall the short-hand notation $\Psi$ in Eq. \eqref{Psi}.
Here and throughout, for a 2-D vector field $W=(W^1, W^2)^{\top}$ ({\it e.g.} $W=\shape \circ \pi(\Psi)$),
we define its curl as $\curl\,W := \curl\,(W^1, W^2, 0)^\top$.
It suffices to check that
\begin{equation}\label{commutation relation for pi curl rot}
\pi \circ \curl \circ \rot \circ \pi (V) = \rot \circ \pi \circ \curl\, V \qquad \text{ for any } V\in T(\po).
\end{equation}
From here, Eq. \eqref{boundary curl 1} implies
\begin{equation}
\pi\big\{\curl^{k+1} (\proj(u\times \omega))\big\} = -\frac{1}{\zeta} \rot\circ\pi \circ\curl^k \,\Psi
+ 2 \rot\circ\pi \circ \curl^k \,(\shape \circ \pi(\Psi)).
\end{equation}

Indeed, we observe
\begin{align*}
\pi \circ \curl \circ \rot \circ \pi \begin{bmatrix}
V^1\\
V^2\\
V^3
\end{bmatrix} = \pi \circ \curl \begin{bmatrix}
-V^2\\
V^1\\
0
\end{bmatrix} = \pi \begin{bmatrix}
-\na_3 V^1\\
-\na_3 V^2\\
0
\end{bmatrix} = \begin{bmatrix}
-\na_3 V^1\\
-\na_3 V^2
\end{bmatrix},
\end{align*}
and
\begin{align*}
\rot \circ \pi \circ \curl \begin{bmatrix}
V^1\\
V^2\\
V^3
\end{bmatrix} = \rot \circ \pi \begin{bmatrix}
\na_2 V^3 -\na_3 V^2\\
\na_3 V^1 - \na_1 V^3\\
\na_1 V^2 - \na_2 V^1
\end{bmatrix} = \rot \begin{bmatrix} -\na _3 V^2\\
\na_3 V^1
\end{bmatrix} = \begin{bmatrix}
-\na_3 V^1\\
-\na_3 V^2
\end{bmatrix},
\end{align*}
since $V^3 =0$.
Therefore, Eq. \eqref{commutation relation for pi curl rot} is proved,
and the {\em identity} in Eq. \eqref{identity for boundary, order reduction}
follows by induction.

Now, in view of the above identity, $J$ can be expressed as
\begin{align}\label{J, order reduced}
J &= \int_\po \big\langle\Big\{  -\frac{1}{\zeta} \rot\circ \pi( \curl^{r-2}\, (\proj (u \times \omega)))\nonumber\\
 &\qquad\qquad+ 2\rot\circ \pi\circ\curl^{r-2}\,\big(\shape \circ \pi\circ\proj(u\times \omega)\big)\Big\}\times \q_r, \n\big\rangle \ddh.
\end{align}
The crucial observation is that only the derivatives up to the $(r-1)$-th order of $u$ are involved.
This is because $\shape$ has a bounded norm in $C^{r-1}$ owing to the assumption of bounded extrinsic geometry,
and $\proj$, $\pi$, and $\rot$ are all smooth operators with the operator norm bounded by a universal constant.
Then we arrive at the following estimates:
\begin{align}\label{J for higher order energy}
|J|&\lesssim \|\curl^{r-2}(u \times \omega)\|_{L^2(\po)} \|\q_r\|_{L^2(\po)} \nonumber\\
&\lesssim \|u\|^2_{H^{r-1}(\po)} \|u\|_{H^r(\po)} \nonumber\\
&\lesssim \Big(\e \|u\|^2_{H^r(\Omega)} + \|u\|^2_{H^{r-1}(\Omega)}\Big)    \Big(\e\|u\|_{H^{r+1}(\Omega)} + \|u\|_{H^r(\Omega)}\Big)\nonumber\\
&\simeq \e^2 \|u\|^2_{H^r(\Omega)}\|u\|_{H^{r+1}(\Omega)} + \e \|u\|^3_{H^r(\Omega)} + \e \|u\|^2_{H^{r-1}(\Omega)}\|u\|_{H^{r+1}(\Omega)} + \|u\|^2_{H^{r-1}}\|u\|_{H^r(\Omega)}\nonumber\\
&\lesssim (\e^2+\e) \|u\|^2_{H^{r+1}(\Omega)} + \|u\|^4_{H^r(\Omega)}.
\end{align}
In the above, the first line follows from the Cauchy-Schwarz inequality,
the second line follows from the argument as for Eq. \eqref{curl Psi},
the third line holds by the Sobolev trace inequality,
and the final line follows by the interpolation and Young's inequalities.

{\bf 5.}
Now, combining the estimates in Steps 2--4 for $I$, $J$, and $K$ (especially  Eqs. \eqref{I for higher order energy}--\eqref{curl Psi}
and \eqref{J for higher order energy}),
Eq. \eqref{eq: higher order energy, basic equation} becomes
\begin{align}
&\frac{1}{2}\frac{d}{dt}\int_{\Omega} |\q_r|^2 \dx + \nu \int_\Omega |\na \q_r|^2\dx \nonumber\\
&\lesssim  \e \|\na^{r+1}u\|^2_{L^2(\Omega)} + \frac{1}{\e}\|u\|^2_{H^r(\Omega)}+ \e \|\na \q_r\|^2_{L^2(\Omega)}
+ (\e+\e^2) \|u\|^2_{H^{r+1}(\Omega)} + (1+\frac{1}{\e})\|u\|^4_{H^r(\Omega)},\,\,
\end{align}
where, in light of Theorem \ref{theorem: div-curl estimate}, $\|\na^{r+1}u\|_{L^2(\Omega)} \simeq \|\na \q_r\|_{L^2(\Omega)} \simeq \|\curl^{r+1}\, u\|_{L^2(\Omega)}$,
and similarly $\|\q_r\|_{L^2(\Omega)} \lesssim \|u\|_{H^r(\Omega)}$.
Then, choosing $\e$ suitably small in comparison with $\nu$ and considering the energy at the $r$-th order:
\begin{equation}
E_r:=\|\q_r\|^2_{L^2(\Omega)},
\end{equation}
we obtain the following differential inequality:
\begin{equation}\label{ODineq:energy}
E_r'(t)\leq E_r'(t)+ \nu E_{r+1} \leq ME_r(t)+ ME_r(t)^2,
\end{equation}
where $M$ depends on $\nu$, $\zeta$, and $\|\two\|_{C^{r-1}(\Omega)}$.

To proceed, consider the auxiliary Cauchy problem for ODE:
\begin{equation}\label{ODE}
\begin{cases}
A'(t) = M\big(A(t)+A(t)^2\big),\\
A(0) = E_r(0)+\eta
\end{cases}
\end{equation}
for arbitrary $\eta >0$. It is solved explicitly by
\begin{equation*}
A(t) =\frac{\big(\eta+E_r(0)\big) e^{Mt}}{1-\big(\eta+E_r(0)\big)\big(e^{Mt}-1\big)},
\end{equation*}
so that, for any $t>0$ before the blowup time:
\begin{equation}\label{blowup time T-star}
T_\star = \frac{1}{M} \log\Big(1+\frac{1}{\eta+E_r(0)}\Big) > 0,
\end{equation}
we see that $A(t) < \infty$.
Comparing the differential inequality \eqref{ODineq:energy} with the ODE in \eqref{ODE},
we find that $E_r(t) \leq A(t)$ for all $0 \leq t < T_\star$.
In particular, since $\eta>0$ is arbitrary,
the upper bound for $A$ (hence for $E_r$) is controlled by $E_r(0) := \|u_0\|_{H^r(\Omega)}^2$ and $M$.
This implies
\begin{equation}\label{E_r(t) bounded}
\sup_{0\leq t <T_\star}E_r(t) \leq C.
\end{equation}

{\bf 6.} It remains to derive the $L^2$--estimate for $\p_t\q_r$.
To this end, we take $\p_t$ to Eq. \eqref{eq: evolution of q_r}, multiply by $\p_t \q_r$, and then integrate over $\Omega$ to obtain
\begin{equation}
\frac{1}{2}\frac{d}{dt}\int_\Omega |\p_t \q_r|^2 \dx - \nu \int_\Omega\langle\Delta\p_t \q_r, \p_t \q_r\rangle \dx + \int_\Omega \langle\curl^r \, \p_t \Psi, \p_t\q_r\rangle \dx  = 0.
\end{equation}
Applying integration by parts and the divergence theorem to the last two terms ({\it cf.} Eq. \eqref{integration by parts, for curl} for curl), we arrive at
\begin{align}\label{tilde I,J,K, higher order energy estimate}
&\frac{1}{2}\frac{d}{dt}\int_\Omega |\p_t \q_r|^2\dx + \frac{\nu}{2} \int_\Omega |\na \p_t \q_r|^2\dx \nonumber\\
&= \nu\int_\po \big\langle (\na \p_t \q_r\cdot\p_t\q_r), \n\big\rangle \ddh + \int_\po \langle \curl^{r-1}\,\p_t \Psi \times \p_t\q_r, \n\rangle\ddh \nonumber\\
&\quad + \int_\Omega \langle\curl^{r-1}\,\p_t\Psi,\curl\,\p_t\q_r\rangle\dx =: \tilde{I}+\tilde{J}+\tilde{K}.
\end{align}

Now, it is crucial to observe the following:
Eq. \eqref{eq: higher order energy, basic equation} differs from Eq. \eqref{tilde I,J,K, higher order energy estimate} only by the time derivatives.
More precisely, if we change variables $(\p_t \q_r, \p_t \Psi)$ in the terms, $\tilde{I}, \tilde{J}$, and $\tilde{K}$, in Eq. \eqref{tilde I,J,K, higher order energy estimate}
to $(\q_r, \Psi)$, then $I,J$, and $K$ in Eq. \eqref{eq: higher order energy, basic equation} are immediately recovered.

Furthermore, since the spatial derivatives commute with $\p_t$,
the integration by parts arguments in Steps 2--5 above all carry through.
Therefore, given that $\|\p_t u(t,\cdot)\|_{L^2(\Omega)}^2 \leq C$ for $u_0 \in H^2(\Omega;\R^3)$ which has been established in Theorem 5.3 of \cite{cq2},
we repeat the arguments above for $(\p_t\q_r, \p_t \Psi)$ to deduce
\begin{equation}\label{Er-2(t) bounded}
\sup_{0\leq t< T_\star}\|\p_t u(t,\cdot)\|^2_{H^{r-2}(\Omega)} \leq C,
\end{equation}
where $C$ depends on $\|u_0\|^2_{H^r(\Omega)}$, $\zeta, \nu, \|\two\|_{C^{r-1}(\Omega)}$, and $T_\star$ that is the same blowup time as in Step 5.

Therefore, combining the estimates in Eqs. \eqref{E_r(t) bounded} and \eqref{Er-2(t) bounded}, we conclude the proof.
\end{proof}

As a corollary, if the initial energy $E_r(0)$ of the fluid is uniformly bounded for all $r\in \mathbb{N}$,
Eq. \eqref{blowup time T-star} implies that there exists a uniform life span $T_\star > 0$ for all levels of the kinetic energy.
Therefore, in view of $C^\infty(\Omega; \R^3)=\bigcap_{r \in \mathbb{N}} H^r(\Omega; \R^3)$, we have

\begin{corollary}\label{cor: smooth solution}
Let $u_0\in C^\infty(\Omega; \R^3)$ be a divergence-free vector field,
and let domain $\Omega \subset \R^3$ be of smooth second fundamental form.
Then there exists $T_\star > 0$ such that the initial-boundary value
problem \eqref{IDC}--\eqref{classical Navier boundary condition}
for the Navier-Stokes equations \eqref{NS equation}
has a smooth solution $u\in C^1([0,T_\star); C^\infty(\Omega; \R^3))$
satisfying the kinematic and Navier boundary conditions \eqref{boundary conditions}.
\end{corollary}

\section{Inviscid Limit}

In this section, we establish the inviscid limit from the Navier-Stokes equations
under the kinematic and Navier boundary conditions \eqref{boundary conditions}
to the Euler equations under the no-penetration condition \eqref{boundary condition for Euler}.
The existence and uniqueness of $u$, the strong solution of the Euler equations \eqref{Euler}
satisfying the no-penetration
boundary condition \eqref{boundary condition for Euler}, have been known ({\it cf}. Ebin-Marsden \cite{EMa}).
We obtain the convergence in the Sobolev spaces $H^r$, $r > \frac{5}{2}$, via strong compactness arguments.

To this end, {\em a priori} estimates of the evolution equations for $u^\nu-u$, {\it i.e.}
the difference between the Navier-Stokes solution and the Euler solution, are required.
In particular, technicalities are involved in the estimates for the higher-order Sobolev norms
of the nonlinear terms, for instance, the iterated curls of $(u^\nu - u) \cdot \na (u^\nu - u)$.
To deal with the nonlinearities, we need to make full use of the incompressibility condition
of $u^\nu$ and $u$, as well as the kinematic and Navier boundary conditions \eqref{boundary conditions}.

The main theorem of this section is stated as follows:

\begin{theorem}[Inviscid Limit]\label{theorem: inviscid limit}
Let $u\in C([0, T_\star); H^{r+1}(\Omega;\R^3)\cap K_2(\Omega))$ be the unique strong solution of the incompressible Euler equations \eqref{Euler}
subject to the no-penetration boundary condition \eqref{boundary condition for Euler}.
Then there exists some $\nu_\star=\nu_\star(T_\star, \int_0^{T_\star} \|u(t,\cdot)\|^2_{H^{r+1}(\Omega)})$ such that,
whenever $0<\nu \leq \nu_\star$, the strong solution of the Navier-Stokes equations \eqref{NS equation}
with the kinematic and Navier boundary conditions \eqref{boundary conditions} exists
in $L^\infty([0,T_\star);H^r(\Omega;\R^3))$ for $r>2$.
Moreover, if $r>\frac{5}{2}$, then there exists a constant $C$ depending only on
$T_\star$, $\|u\|_{L^2([0,T_\star); H^{r+1}(\Omega))}$, and $\|\two\|_{C^r(\po)}$
such that
\begin{equation}
\sup_{0 \leq t <T_\star} \|u^\nu(t,\cdot) - u(t,\cdot)\|_{H^{r}(\Omega)} \leq C \nu^{\frac{1}{3}}.
\end{equation}
In particular, as $\nu \rightarrow 0^{+}$, $u^\nu$ converges to $u$ in $H^{r}(\Omega)$
uniformly in $t\in [0, T_\star)$.
\end{theorem}

\begin{proof}
We divide the arguments into seven steps.

{\bf 1.} First of all,  define
\begin{equation}
\vv := u^\nu - u, \qquad P^\nu = p^\nu-p.
\end{equation}
Our goal is to show that $\vv$ converges to zero in the $H^{r}$--norm, uniformly in $\nu$.
First, subtracting the Euler equations \eqref{Euler} from the Navier-Stokes equations \eqref{NS equation}
yields
\begin{equation}\label{vanishing viscosity, main eqn}
\begin{cases}
\p_t \vv + \big\{(u+\vv)\cdot \na\big\}\vv - \nu \Delta \vv + (\vv \cdot \na) u -\nu \Delta u + \na P^\nu =0,\\
\na \cdot \vv = 0.
\end{cases}
\end{equation}

Next, for
\begin{equation}
\vr:= \curl^r \, \vv \in H^1(\Omega; \R^3),
\end{equation}
noticing that $\curl(\na P^\nu) = 0$, we have
\begin{equation}\label{vr, vanishing viscosity}
\p_t \vr + \curl^r \,\big\{\big((u+\vv)\cdot\na\big)\vv\big\} + \curl^r \,\big\{(\vv\cdot \na)u\big\} - \nu \Delta \vr - \nu \Delta \curl^r\, u =0.
\end{equation}

In view of Theorem \ref{theorem: div-curl estimate}, in order to bound $\vv$ in $H^r$,
it suffices to bound the $L^2$--norm of $\vr$.
To do so, we multiply $\vr$ to Eq. \eqref{vr, vanishing viscosity} and integrate over $\Omega$ to obtain
\begin{align}\label{vanishing viscosity, by parts}
0&= \frac{1}{2} \frac{d}{dt}\int_\Omega |\vr|^2\dx + \int_\Omega \vr \cdot \curl^r\,\big\{\big((u+\vv)\cdot\na\big)\vv\big\} \dx
    +\int_\Omega \vr \cdot \curl^r \,\big\{(\vv\cdot \na)u\big\} \dx\nonumber\\
&\quad  -\nu \int_\Omega \vr \cdot \Delta \vr \dx - \nu \int_\Omega \vr \cdot \Delta\curl^r\, u \dx.
\end{align}

Using integration by parts and the divergence theorem, the fourth term on the right-hand side
of Eq. \eqref{vanishing viscosity, by parts} becomes
\begin{align*}
 -\nu \int_\Omega \vr \cdot \Delta \vr \dx &= -\nu\int_\po (\vr\cdot\na \vr)\cdot\n\ddh + \nu \int_\Omega|\na\vr|^2\dx \\
 &=-\nu \int_\po \two(\vr,\vr)\ddh + \nu \int_\Omega|\na\vr|^2\dx,
\end{align*}
since $\two=-\na\n$.
Meanwhile, the fifth term on the right-hand side of Eq. \eqref{vanishing viscosity, by parts} becomes
\begin{align*}
- \nu \int_\Omega \vr \cdot \Delta\curl^r\, u \dx &= -\nu \int_\po \langle\vr\cdot \na (\curl^r\, u),\n\rangle\ddh + \nu \int_\Omega(\na\vr): (\na\curl^r\,u)\dx,
\end{align*}
where $A:B = \sum_{i,j=1}^3 A_{ij}B_{ji}$ for  $3\times 3$ matrices $A$ and $B$.
Therefore, the following identity is derived from Eq. \eqref{vanishing viscosity, by parts}:
\begin{align}\label{vanishing viscosity, energy estimate}
&\frac{1}{2}\frac{d}{dt}\int_\Omega|\vr|^2 \dx + \nu \int_\Omega |\na\vr|^2\dx \nonumber\\
&= \nu\int_\po \two(\vr,\vr) \ddh - \nu \int_\Omega(\na\vr): (\na\curl^r\,u)\dx -\int_\Omega \big\langle \vr, \,\curl^r\,\big\{\big((u+\vv)\cdot\na\big)\vv\big\} \big\rangle\dx\nonumber\\
&\quad
+\nu \int_{\partial\Omega} \langle \mathbf{V}^\nu_r, \partial_{\mathbf{n}}(\mathbf{curl}^r u) \rangle\,{\rm d}\mathcal{H}^2
 - \int_\Omega \big\langle\vr, \,\curl^r\big\{(\vv\cdot\na)u\big\} \big\rangle\dx\nonumber\\
&= \V^1 + \V^2+\V^3+\V^4+\V^5.
\end{align}
In the following steps, we estimate each of the five terms $\V^1$--$\V^5$.

\smallskip
{\bf 2.} The bound for $\V^1$ is straightforward: As the second fundamental form is bounded in $C^{r-1}(\po)$,
we have
\begin{equation}\label{V1 final}
|\V^1| \lesssim \nu\|\vr\|^2_{L^2(\po)} \simeq \nu \|\vv\|^2_{H^r(\po)} \lesssim {\e}\nu \|\curl^{r+1}\,\vv\|^2_{L^2(\Omega)} + \frac{\nu}{\e} \|\vv\|^2_{H^{r}(\Omega)},
\end{equation}
thanks to the Cauchy-Schwarz inequality, Theorem \ref{theorem: div-curl estimate}, the Young inequality,
and the Sobolev trace inequality.
In addition, the bound for $\V^2$ is also immediate:
\begin{align}\label{V2 final}
|\V^2| &\lesssim \nu \|\na \vr\|_{L^2(\Omega)} \|\na\curl^r\,u\|_{L^2(\Omega)} \nonumber\\
&\simeq \nu \|\curl^{r+1}\,\vv\|^2_{L^2(\Omega)} \|u\|_{H^{r+1}(\Omega)}\nonumber\\
&\lesssim \e \nu \|\curl^{r+1}\,\vv\|^2_{L^2(\Omega)}+ \frac{\nu}{\e} \|u\|^2_{H^{r+1}(\Omega)},
\end{align}
by using the Cauchy-Schwarz and Young inequalities again.

{\bf 3.} Next, we  bound  $\V^3:=-\int_\Omega \langle \vr, \,\curl^r\,\{((u+\vv)\cdot\na)\vv\} \rangle\dx$
by applying a commutator estimate in the spirit of Kato-Ponce \cite{kato-ponce} to the nonlinear convective terms.
For this purpose, we introduce the following abbreviation for the operator:
\begin{equation}
\mathcal{T}=\mathcal{T}(u,\vv) := (u+\vv)\cdot \na.
\end{equation}
Then $\V^3$ can be written as
\begin{align}\label{V3}
\V^3 &= -\int_\Omega \langle \vr, \curl^r\circ \mathcal{T} (\vv)\rangle \dx \nonumber\\
&= -\int_\Omega \langle\vr, \mathcal{T} (\vr) \rangle \dx - \int_\Omega \langle \vr, [\curl^r, \mathcal{T}] \vv \rangle \dx,
\end{align}
where
\begin{equation*}
[\curl^r, \mathcal{T}] = \curl^r\circ \mathcal{T} - \mathcal{T}\circ \curl^r
\end{equation*}
is the commutator of the differential operators.
Then the first term on the right-hand side of Eq. \eqref{V3} equals
\begin{equation}
-\int_\Omega \langle\vr, \mathcal{T} (\vr) \rangle \dx = - \frac{1}{2}\int_\po |\vr|^2 (u + \vv) \cdot \n \ddh = 0,
\end{equation}
thanks to the divergence theorem, the incompressibility of $u$ and $\vv$, and the kinematic boundary conditions for $u$ and $\vv$.

For the second term on the right-hand side of Eq. \eqref{V3}, since $\mathcal{T}$ is a first-order differential operator,
$[\curl^r, \mathcal{T}]$ is of order less than or equal $r$, with the coefficients involving the derivatives of $u$ and $\vv$.
More precisely, we have
\begin{equation}
\big| [\curl^r, \mathcal{T}] \vv \big| \lesssim \sum_{l=1}^r |\na^{[l]} u + \na^{[l]} \vv |\,|\na^{[r+1-l]}\vv|,
\end{equation}
by directly applying the Leibniz rule, where  the schematic symbol $\na^{[l]}$ denotes, as before,
the derivatives up to the $l$-th order.
Next, utilizing the interpolation inequalities, we obtain
\begin{align}\label{xxxx}
|\V^3| &= \Big|\int_\Omega \langle \vr, [\curl^r, \mathcal{T}] \vv \rangle \dx\Big| \nonumber\\
&\lesssim \sum_{l=1}^r \int_\Omega  \left\{|\na^{[r]}\vv| |\na^{[l]}u||\na^{[r+1-l]}\vv| + |\na^{[r]}\vv||\na^{[l]}\vv||\na^{[r+1-l]}\vv| \right\}\dx \nonumber\\
&\lesssim \|\vv\|^2_{H^r(\Omega)} \|\na u\|_{L^\infty(\Omega)} + \|\vv\|_{H^r(\Omega)}\|u\|_{H^r(\Omega)} \|\na \vv\|_{L^\infty} + \|\vv\|_{H^r(\Omega)}^2 \|\na \vv\|_{L^\infty}.
\end{align}
Therefore, in view of the Sobolev embedding $H^{r+1}\emb W^{1,\infty}$ in 3-D for $r>\frac{5}{2}$, we have
\begin{equation}\label{V3_final}
|\V^3| \lesssim \big( 1+\|u\|_{H^{r+1}(\Omega)} \big) \|\vv\|_{H^r(\Omega)}^2 +\|\vv\|^3_{H^r}.
\end{equation}
In particular, we observe that the above upper bound for $\V^3$ involves only the derivatives up to the $r$-th order of $\vv$.

{\bf 4.} Next, to treat the fourth term
$$
\V^4 := \nu \int_{\partial\Omega}\langle \mathbf{V}^\nu_r, \partial_{\mathbf{n}} (\mathbf{curl}^r u)\rangle\,{\rm d}\mathcal{H}^2,
$$
it is crucial to take into account the Navier boundary condition.
Indeed, using Einstein's summation convention and integrating by parts on $\po$, we first obtain
\begin{align}\label{V4, decomposition}
\V^4 &= \nu \int_\po (\vr)^i \na_j (\curl^r\, u)^i \n^j \ddh \nonumber\\
&=-\nu \int_\po \na_j(\vr)^i (\curl^r\, u)^i \n^j \ddh + \nu \int_\po H \langle \vr, \curl^r\,u \rangle \ddh =: \V^{4,1} + \V^{4,2},
\end{align}
where $H=-\na_j \n^j$ is the mean curvature of boundary $\po$.
Here, by the boundedness of the second fundamental form, the second term is bounded by
\begin{align}\label{V4,2}
|\V^{4,2}| &\lesssim \nu \big(\|\vv\|^2_{H^r(\po)} + \|u\|^2_{H^r(\po)}\big) \nonumber\\
&\lesssim \e \nu \big(\|\vv\|^2_{H^{r+1}(\Omega)} + \|u\|^2_{H^{r+1}(\Omega)}\big)
  + \frac{\nu}{\e}\big(\|\vv\|^2_{H^r(\Omega)} + \|u\|^2_{H^r(\Omega)}\big).
\end{align}

For $\V^{4,1}$, we add and subtract $\na_i (\vr)^j$ from the integrand to obtain
\begin{align}
\V^{4,1} &= -\nu \int_\po \Big\{ \na_j (\vr)^i-\na_i(\vr)^j\Big\} (\curl^r\, u)^i \n^j \ddh \nonumber\\
&\quad  + \nu \int_\po \na_i(\vr)^j (\curl^r\, u)^i \n^j \ddh =: \V^{4,1}_{\rm a} + \V^{4,1}_{\rm s}.
\end{align}
By the incompressibility condition, the integral vanishes for $i=j$ so that it suffices to consider $i \neq j$ in the summation.
In this case, denoting by $k$ the index in $\{1,2,3\}$ different from $i, j$, we have
\begin{equation}\label{V4,1_a}
\V^{4,1}_{\rm a} = \sigma \nu \int_\po (\curl^{r+1}\, \vv)^k (\curl^r \, u)^i \n^j \ddh,
\end{equation}
where $\sigma\in \{1, -1\}$ is a sign.

To proceed, we first establish the following {\em claim}: For the tangential projection
$\pi: T\R^3 \rightarrow T(\po)$ as before,
the iterated curls of $\vv$ satisfy the following non-Navier slip-type boundary condition:
\begin{equation}\label{claim: boundary condition for curls of v}
\pi \circ \curl^{r+1} \, \vv = -\frac{1}{\zeta} \rot\circ \pi (\curl^r \vv) + 2 \rot \circ \pi \circ \curl^r \circ \shape \circ \pi (\vv) \qquad \text{ on } \po.
\end{equation}

The proof of the above {\em claim} goes by induction on $r$, similar to the arguments in Step $4$ in the proof
of Theorem \ref{theorem: higher order energy estimate}.
Indeed, for $r=0$, it reduces to the Navier boundary condition \eqref{boundary conditions}; thus, we assume the result first for $r$ and then prove it for $r+1$.
Taking a moving frame $\{\na_1, \na_2, \na_3\}$ adapted to surface $\po$ such that $\na_1,\na_2 \in T(\po)$ and $\na_3 = \n$,
by the induction hypothesis, we have
\begin{align*}
\pi (\curl^{r+1} \, \vv) &= \pi\circ \curl \big( -\frac{1}{\zeta} \rot \circ \pi (\curl^{r-1}\,\vv) + 2 \rot \circ \pi \circ \curl^{r-1}\circ\shape \circ \pi (\vv)\big)\\
&=\begin{bmatrix}
-\na_3 \Big( -\frac{1}{\zeta} \rot \circ \pi (\curl^{r-1}\,\vv) + 2 \rot \circ \pi \circ \curl^{r-1}\circ\shape \circ \pi (\vv)\Big)^2\\
\na_3\Big( -\frac{1}{\zeta} \rot \circ \pi (\curl^{r-1}\,\vv) + 2 \rot \circ \pi \circ \curl^{r-1}\circ\shape \circ \pi (\vv)\Big)^1
\end{bmatrix}.
\end{align*}
Since $\rot ((V^1, V^2)^\top) = (-V^2, V^1)^\top$ for any 2-D vector $V$,
the above equalities yield
\begin{align}
\pi(\curl^{r+1} \, \vv) &= \begin{bmatrix}
-\na_3 \Big(-\frac{1}{\zeta} (\curl^{r-1}\, \vv)^1 + 2 (\curl^{r-1}\circ \shape \circ\pi(\vv))^1\Big) \\
-\na_3 \Big( -\frac{1}{\zeta} (\curl^{r-1}\,\vv)^2 + 2 (\curl^{r-1}\circ \shape \circ\pi(\vv))^2 \Big)
\end{bmatrix}\nonumber\\
&= -\frac{1}{\zeta} \Big\{-\na_3 (\curl^{r-1}\, \vv)\Big\} -2\na_3 \Big\{\curl^{r-1}\circ \shape \circ\pi(\vv)\Big\}\nonumber\\
&= -\frac{1}{\zeta} \Big\{ \rot \circ \pi \circ \curl \, (\curl^{r-1}\, \vv)  \Big\} + 2 \Big\{ \rot\circ \pi\circ \curl\, (\curl^{r-1}\circ \shape \circ\pi(\vv)) \Big\} \nonumber\\
&=  -\frac{1}{\zeta} \Big\{\rot \circ \pi \circ \curl^r\, \vv\Big\}  + 2\Big\{\rot \circ \pi \circ \curl^r \circ \shape \circ \pi (\vv)\Big\}.
\end{align}
Thus,  {\em claim} \eqref{claim: boundary condition for curls of v} is proved.

The above {\em claim} shows that the iterated curls of $\vv$ can be expressed on boundary $\po$ by the derivatives of $\vv$ up to the $r$-th order,
together with the bounded operators $\rot$, $\shape$, and $\pi$.
Thus, Eq. \eqref{V4,1_a} becomes
\begin{align*}
\V^{4,1}_{\rm a} &= \sigma \nu \int_\po \Big\{ -\frac{1}{\zeta} [\rot \circ \pi(\curl^r\,\vv)]^k (\curl^r\, u)^i \n^j  \Big\} \ddh\\
&\quad +2\sigma\nu \int_\po \Big\{[\rot\circ\pi\circ\curl^r\circ\shape\circ\pi(\vv)]^k(\curl^r \, u)^i \n^j\Big\} \ddh,
\end{align*}
where the non-zero contributions come only from the tangential components of $\curl^{r+1} \vv$, and $\n^j=0$ unless $j=3$ in the moving frame $\{\na_1, \na_2,\na_3\}$,
which forces $k \in \{1,2\}$.
We then obtain the following estimate:
\begin{align}\label{V4,1_a final}
|\V^{4,1}_{\rm a}| &\lesssim \nu \Big(\|\vv\|_{H^r(\po)}^2 + \|u\|^2_{H^r(\po)}\Big)\nonumber\\
&\lesssim \nu \Big( \e\|\curl^{r+1}\,\vv\|^2_{H^{r+1}(\Omega)} + \|\vv\|^2_{H^r(\Omega)} + \e \|u\|^2_{H^{r+1}(\Omega)} + \|u\|^2_{H^r(\Omega)} \Big).
\end{align}

On the other hand, for $\V^{4,1}_{\rm s}$, we integrate by part once more to obtain
\begin{equation}
\V^{4,1}_{\rm s} = -\nu \int_\po (\vr)^j (\curl^r \, u)^i \na_i\n^j \ddh = \nu \int_\po \two(\vr, \curl^r\, u)\ddh.
\end{equation}
This is because ${\rm div}\circ \curl^r\, u = \curl^r\circ {\rm div}\, u =0$, and $\two = - \na \n$
by the definition of the second fundamental form.
By assumption, $\|\two\|_{L^\infty(\po)} \leq C$ so that
\begin{align}\label{V4,1_s}
|\V^{4,1}_{\rm s}| &\lesssim \nu \|\vr\|_{L^2(\po)}\|\curl^r\,u\|_{L^2(\po)}\nonumber\\
&\lesssim \e\nu \Big(\|\curl^{r+1}\|^2_{L^2(\Omega)} + \|u\|^2_{H^{r+1}(\Omega)}\Big) + \frac{\nu}{\e} \Big(\|\vv\|^2_{H^r(\Omega)} + \|u\|^2_{H^r(\Omega)}\Big).
\end{align}
Thus, putting together Eqs. \eqref{V4, decomposition}, \eqref{V4,1_a final}, \eqref{V4,1_s}, and \eqref{V4,2}, we conclude
\begin{equation}\label{V4 final}
|\V^4| \lesssim \e\nu \Big(\|\curl^{r+1}\, \vv\|^2_{L^2(\Omega)} + \|u\|^2_{H^{r+1}(\Omega)}\Big) + \frac{\nu}{\e} \Big(\|\vv\|^2_{H^r(\Omega)} + \|u\|^2_{H^r(\Omega)}\Big).
\end{equation}

{\bf 5.} The estimate for $\V^5:=-\int_\Omega \big\langle \curl^r\, \vv , \curl^r \,(\vv \cdot \na u)\big\rangle \dx$
proceeds by the Leibniz rule:
\begin{align*}
|\V^5| &\lesssim \int_\Omega |\vr| \times \Big(\sum_{l=0}^r |\na^{[l]}\vv| |\na^{[r-l+1]}u |\Big) \dx \\
&\lesssim \int_\Omega |\na^{[r]}\vv| \Big(|\na^r\vv||\na u| + |\vv| |\na^{r+1} u| \Big) \dx \\
&\lesssim \|\vv\|^2_{H^r(\Omega)} \|\na u\|_{L^\infty(\Omega)} + \|\vv\|_{H^r(\Omega)} \|\vv\|_{L^\infty(\Omega)} \|u\|_{H^{r+1}(\Omega)} \\
&\lesssim \|\vv\|^2_{H^r(\Omega)} \| u\|_{H^3(\Omega)} + \|\vv\|_{H^r(\Omega)} \|\vv\|_{H^2(\Omega)} \|u\|_{H^{r+1}(\Omega)},
\end{align*}
where the last two lines hold by the interpolation inequality and the Sobolev embedding $H^2(\R^3) \emb L^\infty(\R^3)$.
Thus, we have
\begin{equation}\label{V_5 final}
|\V^5| \lesssim \|\vv\|^2_{H^r(\Omega)} \|u\|_{H^{r+1}(\Omega)},
\end{equation}
whenever $r \geq 2$.

{\bf 6.} Now, putting all the estimates for $\V^1$ -- $\V^5$ in Eqs. \eqref{V1 final}--\eqref{V2 final}, \eqref{V3_final}, and \eqref{V4 final}--\eqref{V_5 final} together,
we obtain
\begin{align}
&\frac{1}{2}\frac{d}{dt} \int_\Omega |\vr|^2 \dx + \nu \int_\Omega |\na \vr|^2 \dx \nonumber\\
&\lesssim \e\nu \|\curl^{r+1} \, \vv\|^2_{L^2(\Omega)} + \big(\e \nu + \frac{\nu}{\e} + \|u\|_{H^{r+1}(\Omega)}\big) \|\vv\|^2_{H^r(\Omega)}
  +\|\vv\|^3_{H^r(\Omega)} + \big( \e\nu+\frac{\nu}{\e}\big) \|u\|^2_{H^{r+1}(\Omega)}.
\end{align}
Thus, we can choose $\e>0$ to be small so that
\begin{equation}\label{before gronwall}
\frac{d}{dt}\int_\Omega |\vr|^2\dx + \nu \|\curl^{r+1}\vv\|^2_{L^2(\Omega)}
\leq C_1 (1+\|u\|_{H^{r+1}(\Omega)})\|\vv\|^2_{H^r(\Omega)} +\|\vv\|^3_{H^r(\Omega)} + C_2 \nu\|u\|^2_{H^{r+1}(\Omega)},
\end{equation}
where constants $C_1$ and $C_2$ depend on $\nu_0$ and $\e$.
Here, it is crucial to choose $\e$ that depends only on $\|\two\|_{C^{r}(\po)}$, {\em independent of $\nu$}.

As a consequence, the energy function  $E(t):=\|\vv(t,\cdot)\|_{H^r(\Omega)}^2$
satisfies the differential inequality:
\begin{equation}\label{ODineq, without additional assumption}
E'(t) \leq C_3 E(t)^{3/2}
 + C_4\nu,
\end{equation}
where $C_3=C_3(C_1, \|u\|_{H^{r+1}(\Omega)})$ and $C_4=C_2\|u\|_{H^{r+1}(\Omega)}^2$ by interpolation.

{\bf 7.}
Finally, in order to show that the lifespan of strong solutions of the Navier-Stokes equations
are no less than $T_\star$ uniformly in $\nu \in (0, \nu_\star)$
and to derive the inviscid limit, for any fixed $T\in (0,T_\star)$,
we define
\begin{equation}
\begin{cases}
\alpha(t):= \exp\big\{-C_1 t-C_1\int_0^t \|u(s,\cdot)\|_{H^{r+1}(\Omega)}\,\dd s\big\},\\
F(t):=C_2\alpha(t) \|u(t,\cdot)\|^2_{H^{r+1}(\Omega)},\\
G:= C_1 \alpha(T)^{-1} = {\rm const.},\\
\Phi(t):=\alpha(t) \|\vv(t,\cdot)\|_{H^{r}(\Omega)}^2.
\end{cases}
\end{equation}
Then $\Phi$ satisfies the ordinary differential inequality:
\begin{equation}\label{ODineq for Phi}
\Phi'(t) \leq \nu F(t) + G\Phi(t)^{3/2} \qquad \text{ on } [0,T].
\end{equation}
For some parameter $\beta >0$ to be chosen later, we divide Eq. \eqref{ODineq for Phi} by $\big(1+\beta\Phi(t)\big)^{3/2}$ to obtain
\begin{equation*}
\frac{\Phi'(t)}{(1+\beta\Phi(t))^{3/2}} \leq \nu F(t) + \frac{G}{\beta^{3/2}}.
\end{equation*}
Thus, by integrating from $0$ to $t \in [0,T]$, we obtain the estimate:
\begin{equation*}
\frac{2}{\beta} \Big(1-\frac{1}{\sqrt{1+\beta\Phi(t)}}\Big) \leq \nu \int_0^T F(t)\dd t + \frac{GT}{\beta^{3/2}},
\end{equation*}
that is,
\begin{equation}\label{eqn: rhs larger than a half}
\frac{1}{\sqrt{1+\beta\Phi(t)}} \geq 1 - \frac{\beta \nu}{2}\int_0^T F(t)\dd t - \frac{GT}{2\sqrt{\beta}}.
\end{equation}

Notice that the right-hand side of Eq. \eqref{eqn: rhs larger than a half} is maximized when the two negative terms
are equal:
\begin{equation}
\beta = \Big(\frac{GT}{\nu\int_0^T F(t)\dd t}\Big)^{2/3}.
\end{equation}
Then the right-hand side equals
\begin{equation*}
1- (GT)^{\frac{2}{3}}\Big(\nu \int_0^TF(t)\,{\rm d}t\Big)^{\frac{1}{3}}.
\end{equation*}
It is bigger than or equal to $\frac{1}{2}$ if and only if
\begin{equation*}
\nu \leq \Big(8(GT)^2 \int_0^TF(t)\,{\rm d}t\Big)^{-1}.
\end{equation*}
On the other hand, by Eq. \eqref{eqn: rhs larger than a half}, $\frac{1}{\sqrt{1+\beta\Phi(t)}} \geq \frac{1}{2}$
for all $0 \leq t \leq T$; that is, $\Phi(t) \leq \frac{3}{\beta}$ on $[0,T]$.

In summary, we have established the following: If we set
\begin{equation*}
\nu_\star = \Big(8 G^2T_\star^2 \int_0^{T_\star} F(t) \, {\rm d}t\Big)^{-1}.
\end{equation*}
then, whenever $0< \nu \leq \nu_\star$,
\begin{equation}
\Phi(t) \leq 3 \Big(\frac{\int_0^T F(t)\dd t}{GT}\Big)^{2/3} \nu^{2/3},
\end{equation}
which is equivalent to
\begin{align}
&\sup_{t \in [0,T]}\|u^\nu (t,\cdot)-u(t,\cdot)\|^2_{H^r(\Omega)} \nonumber\\
&\leq 3\exp\Big\{C_1 T + C_1\int_0^T \|u(s,\cdot)\|_{H^{r+1}(\Omega)}\,{\rm d}s \Big\}
\,\Big( \frac{\int_0^T F(t)\,{\rm d}t}{GT} \Big)^{\frac{2}{3}}\nu^{\frac{2}{3}}.
\end{align}
This holds for any $t\in [0,T]$, where $T$ is an arbitrary number in $(0,T_\star)$.
Therefore, the Navier-Stokes solution $u^\nu$ does not blow up on $[0,T]$ in the $H^r$--norm in space,
provided that $0 < \nu\leq\nu_\star$. This completes the proof.
\end{proof}

To conclude this section, we now give the following three remarks.

\begin{remark}
A key point of Theorem {\rm \ref{theorem: inviscid limit}} is that the strong solutions $u^\nu$
to the Navier-Stokes equations do not blow up before $T_\star$ in $H^r$,
where $T_\star$ is the lifespan of the corresponding Euler equations in $H^{r+1}$ for $r>\frac{5}{2}$.
The arguments {\rm (}Step {\rm 7} of the proof{\rm )} are adapted from  \S\,{\rm 1} in Constantin \cite{c1},
in which the case of periodic boundary conditions are treated.
This does not directly follow from our proof of Theorem {\rm \ref{theorem: higher order energy estimate}}.
In fact,  constant $M$ is proportional to $\nu^{-1}$ in Eq. \eqref{ODineq:energy},
so that the lifespan for the Navier-Stokes equations {\rm (}in $H^{r+1}$ in space{\rm )} is proportional
to viscosity $\nu$, which goes to zero in the vanishing viscosity limit.
\end{remark}

\begin{remark} In Theorem {\rm \ref{theorem: inviscid limit}}, the rate of convergence in the inviscid limit
is $\mathcal{O}(\nu^{1/3})$. It can be improved to $\mathcal{O}(\sqrt{\nu})$, provided that $\{\na u^\nu - \na u\}$ is uniformly
bounded in space-time.
Moreover, in this case, the $H^{r+1}$--norm of $u^\nu$ is also close to that of $u$ in the average in time.

\begin{proposition}\label{proposition: inviscid limit under additional assumption}
Let $u^\nu, u, \nu^0, T_\star$, and $r$ be as in Theorem {\rm \ref{theorem: inviscid limit}}.
In addition, suppose that $\{\na u^\nu - \na u\}$ is uniformly bounded in $L^\infty([0,T_\star) \times \Omega; \R^3)$.
Then there exists a constant $C$, depending only on  $T_\star$, $\|u\|_{L^2([0,T_\star); H^{r+1}(\Omega))}$, and $\|\two\|_{C^r(\po)}$,
such that
\begin{equation}
\sup_{0 \leq t <T_\star} \|u^\nu(t,\cdot) - u(t,\cdot)\|_{H^{r}(\Omega)} \leq C\sqrt{\nu}.
\end{equation}
In particular, as $\nu \rightarrow 0^{+}$, $u^\nu$ converges to $u$ in $H^{r}(\Omega)$ uniformly in time. In addition,
\begin{equation}
\int_0^{T} \|u^\nu(t,\cdot) - u(t,\cdot)\|_{H^{r+1}(\Omega)} \leq C \qquad \text{ for any } T \in [0,T_\star).
\end{equation}
\end{proposition}

\begin{proof}
The proof follows essentially from the arguments in Theorem \ref{theorem: inviscid limit} above,
{\it i.e.} by considering the evolution equation for $v^\nu :=u^\nu - u$.
We only emphasize the differences.

Indeed, starting from Eq. \eqref{vanishing viscosity, energy estimate}, we estimate the terms, $\V^1, \V^2, \V^4$, and $\V^5$,
as in Steps 1--2 and  4-- 5 in the proof of Theorem \ref{theorem: inviscid limit}.
The only difference occurs in Step 3. Recall Eq. \eqref{xxxx} therein:
\begin{equation*}
|\V^3| \lesssim  \|\vv\|^2_{H^r(\Omega)} \|\na u\|_{L^\infty(\Omega)} + \|\vv\|_{H^r(\Omega)}\|u\|_{H^r(\Omega)} \|\na \vv\|_{L^\infty} + \|\vv\|_{H^r(\Omega)}^2 \|\na \vv\|_{L^\infty}.
\end{equation*}
Under the additional assumption, $\|\na \vv\|_{L^\infty} \leq C$ so that
\begin{equation}
|\V^3| \lesssim \big(1+\|u\|_{H^r(\Omega)}\big) \|\vv\|^2_{H^r(\Omega)}.
\end{equation}
As a consequence, by choosing $\e$ suitable small,  estimate \eqref{before gronwall} in Step 6 can be improved to
\begin{equation}\label{before gronwall new}
\frac{d}{dt}\int_\Omega |\vr|^2\dx + \nu \|\curl^{r+1}\vv\|^2_{L^2(\Omega)} \leq C_1\Big(1+\|u\|_{H^{r+1}(\Omega)}\Big)\|\vv\|^2_{H^r(\Omega)} + C_2 \nu\|u\|^2_{H^{r+1}(\Omega)},
\end{equation}
which does not contain the cubic terms in $\|\vv\|_{H^r}$.

From here, the usual Gronwall inequality yields
\begin{align}\label{Gronwall}
\|\vv (t,\cdot)\|^2_{H^r(\Omega)} &\leq \|\vv(0,\cdot)\|^2_{H^r(\Omega)} \exp \Big\{C_1 t+\int_0^t \|u(s,\cdot)\|_{H^{r+1}(\Omega)}\,{\rm d}s \Big\} \nonumber\\
&\quad + C_2\nu \int_0^t \Big\{ \exp\Big(C(t-s) + \int_s^t \|u(\tau,\cdot)\|_{H^{r+1}(\Omega)} \,{\rm d}\tau \Big) \|u(s,\cdot)\|^2_{H^{r+1}(\Omega)} \Big\} \,{\rm d}s \nonumber\\
&= C_2\nu \int_0^t \Big\{ \exp\Big(C(t-s) + \int_s^t \|u(\tau,\cdot)\|_{H^{r+1}(\Omega)} \,{\rm d}\tau \Big) \|u(s,\cdot)\|^2_{H^{r+1}(\Omega)} \Big\} \,{\rm d}s
\end{align}
for any $t \in [0, T_\star)$.
This is because $\vv(0,\cdot)=0$, since the Navier-Stokes and the Euler solutions have the same initial data.
Thus, for some constant $C_6=C_6(\nu^0, \|\two\|_{C^{r}(\Omega)}, T_\star, \|u\|_{L^2(0,T_\star; H^{r+1}(\Omega))})$,
\begin{equation}\label{vanishing viscosity: the final estimate}
\sup_{t \in [0, T_\star)}\|\vv(t,\cdot)\|_{H^r(\Omega)} \leq C_6 \sqrt{\nu} \longrightarrow 0 \qquad \text{ as } \nu \rightarrow 0^{+}.
\end{equation}

Finally, integrate Eq. \eqref{before gronwall new} with Eq.  \eqref{Gronwall} substituted into the right-hand side.
In this way, we find a constant $C_7$ with the same dependence as $C_6$ such that
\begin{equation}
\sup_{t \in [0, T_\star)}\int_{0}^t \|\curl^{r+1}\,\vv(s,\cdot) \|^2_{L^2(\Omega)} \,{\rm d} s \leq C_7.
\end{equation}
Therefore, in view of Theorem \ref{theorem: div-curl estimate} and Eq. \eqref{vanishing viscosity: the final estimate}, we have
\begin{equation}
\sup_{t \in [0, T_\star)}\int_{0}^t \|\vv(s,\cdot) \|^2_{H^{r+1}(\Omega)} \,{\rm d} s \leq C_8=C_8(\nu^0, \|\two\|_{C^{r}(\Omega)}, T_\star, \|u\|_{L^2(0,T_\star; H^{r+1}(\Omega))}),
\end{equation}
which completes the proof.
\end{proof}
\end{remark}

\begin{remark}
Combining the results in \S {\rm 4}--\S {\rm 5} together,
we have established the existence of strong solutions in $H^{r+1}$ for $r>\frac{5}{2}$ of the Navier-Stokes equations,
while the inviscid limit has been proved in $H^r$.
Therefore, it remains an open question whether the inviscid limit holds or fails {\rm (}{\it e.g.}, due to the development of boundary layers{\rm )}
in $H^{r+1}$, {\it i.e.} the highest order the spatial regularity of the strong solutions.
\end{remark}

\section{Remarks on the Non-Navier Slip-type Boundary Condition}

In the introduction (\S 1), a modified version of the Navier boundary condition, which is originally introduced
by Bardos \cite{bardos} and Solonnikov-\u{S}\u{c}adilov \cite{solonnikov},
has been briefly discussed.
Physically, it describes the phenomenon that the tangential part of the normal vector field of the Cauchy stress tensor
is uniformly vanishing, and it agrees with the Navier boundary condition if and only if boundary $\po$ is flat.
Together with the kinematic boundary condition, we have
\begin{equation}\label{bc for ns}
u^\nu \cdot \n =0, \quad
\omega^\nu \times \n =0 \qquad\, \text{ on } \po.
\end{equation}
The second line is referred to as
the non-Navier slip-type boundary condition.

In  Beir\~{a}o da Veiga-Crispo \cite{v3,v4},
the inviscid limit problem is analyzed for the Navier-Stokes equations subject to the boundary conditions  \eqref{bc for ns}
and the Euler equations subject to the no-penetration boundary condition  \eqref{boundary condition for Euler}.
In this section, we write $K$ for the {\em Gauss curvature} of surface $\p\Omega$.
We first introduce the following notions (see also Definitions 2.1 and 2.3 in \cite{v4}):

\begin{definition}\label{definition of admissible data, Beirao da Veiga}
For the non-Navier slip-type boundary conditions \eqref{bc for ns}, we say that
\begin{enumerate}
\item[\rm (i)]
$u_0\in C^\infty(\Omega; \R^3)$ is an admissible initial data if $\na \cdot u_0 =0$ in the closure $\overline{\Omega}$,
as well as $u_0 \cdot \n =0$ and $\omega_0 \times \n=0$ on $\po${\rm ;}
\item[\rm (ii)]
The inviscid limit $u^\nu \rightarrow u$ holds ``strongly'' in $L^p([0,T]; W^{s,q}(\Omega; \R^3))$ for some  $T>0$, $p,q \geq 1$, and $s>1$,
if the convergence holds with respect to the strong topology on $L^p([0,T]; W^{s,q}(\Omega; \R^3))$.
\end{enumerate}
\end{definition}

In particular, we notice that, if $u^\nu \rightarrow u$ strongly, $\omega^\nu \times \n =0$ on $[0,T] \times \po$
implies $\omega \times \n =0$ on $[0,T] \times \po$.
This is termed as the ``{\em persistence property}'' in \cite{v3, v4}.
In the presence of such a property, the following non-convergence result is established by Beir\~{a}o da Veiga-Crispo,
first by  considering a special example on $\mathbb{S}^2$ in \cite{v3} and
then proved in full generality via computations of the principal curvatures on $\po$ in local coordinates in \cite{v4}:

\begin{theorem}[Beir\~{a}o da Veiga-Crispo, \cite{v3, v4}]\label{thm: veiga-crispo}
Let $\Omega\subset \R^3$ be a bounded regular domain.
Let the admissible initial data $u_0$ be given for the initial-boundary value problem \eqref{NS equation} and
\eqref{bc for ns}
and problem \eqref{Euler}--\eqref{boundary condition for Euler}
such that the following condition holds{\rm :}
\begin{equation}\label{geometric condition}
\oo_0(x_0) \neq 0 \qquad \text{ for some } x_0 \in \p\Omega \text{ such that } K(x_0) \neq  0.
\end{equation}
Then, for arbitrary $\delta>0$, $p, q \geq 1$ and $s>1$, the ``strong'' inviscid limit fails{\rm :}
\begin{equation}
 u^\nu \nrightarrow u \qquad \text{ in } L^p([0,\delta]; W^{s,q}(\Omega; \R^3)).
\end{equation}
\end{theorem}

This theorem says that, if the initial vorticity vanishes somewhere on the curved part of the boundary,
{\it i.e.} the Gauss curvature is non-vanishing at this point,
then the ``strong'' inviscid limit fails in an arbitrarily short time interval,
so that  the Prandtl boundary layers must be developed.
Heuristically, this is due to the incompatibility of the vorticity directions of the slip-type boundary conditions
in the limiting process $\nu \rightarrow 0^{+}$.

Now we give an alternative proof of Theorem \ref{thm: veiga-crispo}, which avoids the computations in local coordinates
on $\po$ as in \cite{v3, v4}.
This offers a new, global perspective for the above theorem,
and the proof makes essential use of the properties of Lie derivatives in $\R^3$.

\begin{proof}
First of all, following the original proof
in \cite{v3, v4}, we consider the inviscid vorticity equation,
which is obtained by taking the curl of the Euler equation \eqref{Euler}:
\begin{equation}
\p_t \omega + (u \cdot \na)\omega - (\omega \cdot \na)u = 0.
\end{equation}
Then taking the cross product with the outer unit normal $\n: \po \rightarrow \mathbb{S}^2$ leads to
\begin{equation}\label{eq: evolution for omega times n}
\p_t(\omega \times \n) +\big\{(u\cdot\na) \omega -(\omega \cdot\na) u\big\}\times \n = 0.
\end{equation}
It is observed by Xiao-Xin ({\it cf.} Corollary 8.3 in \cite{xin1}) that a necessary condition
for the ``strong'' inviscid limit in the time interval $(0,\delta)$ is:
\begin{equation}
\omega \times \n \equiv 0 \qquad \text{ on } (0,\delta)\times \po,
\end{equation}
that is, the persistence property is verified ({\it cf.} Definition \ref{definition of admissible data, Beirao da Veiga}).
Thus, if the ``strong'' inviscid limit were valid, then  Eq. \eqref{eq: evolution for omega times n}  implies that
$\{(u\cdot\na) \omega -(\omega \cdot\na) u\}\times \n = 0$ for all time.
In particular, sending $t \rightarrow 0^{+}$, the following condition must be fulfilled:
\begin{equation}\label{eq for u0 and w0}
\big\{(u_0\cdot\na) \omega_0 -(\omega_0 \cdot\na) u_0\big\}\times \n = 0 \qquad \text{ on } \po.
\end{equation}

Our crucial observation is that the expression in the bracket in \eqref{eq for u0 and w0}
coincides
with the Lie bracket of the vector fields $u_0, \omega_0 \in T\R^3$:
\begin{equation}\label{express as lie bracket}
(u_0\cdot\na) \omega_0 -(\omega_0 \cdot\na) u_0 = [u_0, \oo_0].
\end{equation}

To prove Eq. \eqref{express as lie bracket}, let $V=\sum_{i=1}^3 V^i\p_i$ and $W=\sum_{j=1}^3 V^j\p_j$ be two smooth vector fields
in $T\R^3$, where $\{\p_1,\p_2, \p_3\}$ denotes the canonical Euclidean frame.
We follow the convention in differential geometry to identify vector fields with first-order differential operators.
Thus, the Lie bracket of $V$ and $W$ can be computed as
\begin{align*}
[V,W] = VW-WV
= \Big(V^i\p_i (W^j\p_j) - W^j\p_j(V^i\p_i) \Big)
=: (V\cdot\na) W - (W \cdot \na) V,
\end{align*}
which verifies the above identity. In the above, the Einstein summation convention is adopted.

To proceed, we recall that the Lie bracket of two vector fields on a differentiable manifold
equals the {\em Lie derivative} (denoted by $\lie$) of one vector field along the other:
\begin{equation}\label{lie bracket and lie derivative}
[u_0, \oo_0] = \lie_{u_0} \omega_0.
\end{equation}
This result is standard in the differentiable manifold theory, which can be found in
the classical texts ({\it cf.}
do Carmo \cite{docarmo}).
Furthermore, the  Lie derivative at a point $x_0 \in \p\Omega$ is given by
\begin{equation}\label{def for Lie derivative}
\lie_{u_0}\omega_0 (x_0):= \lim_{s \rightarrow 0^{+}} \frac{(\theta_{s})^{\ast} \{\omega_0({\theta_s(x_0)})\}- \omega_0(x_0)}{s},
\end{equation}
where $\{\theta_s(x): s\geq 0, x\in \R^3\}$ is the one-parameter subgroup defined by the following ODE:
\begin{equation}\label{ODE for integral flow}
\begin{cases}
\frac{d}{ds} \theta_s(x) = u_0(\theta_s(x))\qquad \text{ for all } s \geq 0, \: x \in \R^3,\\
\theta_0(x) = x \qquad\qquad\quad\quad\,\, \text{ for all } x \in \R^3.
\end{cases}
\end{equation}
In other words, trajectory $\{\theta_s(x)\}_{s \geq 0}$ is the {\em integral curve} of the vector field $u_0$ emanating
from point $x$.
Also,
$(\theta_s)^\ast$ denotes the pullback operation under map $\theta_s$.

By the admissibility of the initial data,
$\oo_0 \in T(\po)^\perp$ is orthogonal to the boundary
since $\oo_0 \times \n =0$,
and $u_0 \in T(\po)$ is tangential to the boundary because of the kinematic boundary condition: $u_0 \cdot \n = 0$.
The expression on the right-hand side of Eq. \eqref{lie bracket and lie derivative} is well-defined
since $\oo_0$ is a vector field defined along the manifold $\po$.
In geometric terminologies, it means that $\oo_0 \in \iota^*T\R^3$,
where $\iota: \po \emb \R^3$ is the embedding of Riemannian submanifold,
and $\iota^*T\R^3$ is the pullback vector bundle.
This enables us to take the Lie derivative on $\oo_0$ along any vector field ({\it e.g.} $u_0$)
tangent to $\po$.

Now, by the assumptions, there is a point $x_0 \in \po$ such that $\oo_0(x_0) \neq 0$ and $K(x_0) \neq 0$.
Owing to the non-vanishing curvature, there exists some small neighbourhood $U \subset \po$ of $x_0$
such that every smooth curve $\gamma: (-\delta, \delta) \rightarrow \po$ satisfying $\gamma(0)=x_0$ is not a straight line
segment in $\R^3$.
In addition, as vorticity $\oo_0$ is non-vanishing at $x_0$, we have
\begin{equation}
\langle u_0, \dot{\gamma} \rangle\neq 0 \qquad \text{ on } U.
\end{equation}

On the other hand, using the definition of the Lie derivative in terms of the integral curve, {\it i.e.} \eqref{def for Lie derivative},
we have
\begin{equation}
\lie_{\dot{\gamma}} \omega_0 \neq 0 \qquad \text{ on } T(\po).
\end{equation}
This is because the parallel-transport of $\omega_0$ along $\gamma$ cannot be obtained by a Euclidean translation,
so that $(\theta_{s})^{\ast} \{\omega_0({\theta_s(x_0)})\} - \omega_0(x_0) \neq 0$ in Eq. \eqref{def for Lie derivative}.
Therefore, we conclude that $\lie_{\langle u_0, \dot{\gamma}\rangle \dot{\gamma}} \omega_0 \neq 0$ on $T(\po)$, from which it follows
\begin{equation}
\lie_{u_0}\oo_0 (x_0) \times \n(x_0) \neq 0 \qquad \text{ in } \R^3.
\end{equation}
This contradicts Eq. \eqref{eq for u0 and w0}.
\end{proof}

\smallskip

\noindent
{\bf Acknowledgement}.
The research of Gui-Qiang G. Chen  was supported in part by
the UK
Engineering and Physical Sciences Research Council Award
EP/E035027/1 and
EP/L015811/1, and the Royal Society--Wolfson Research Merit Award (UK).
The research of Siran Li was supported in part by the UK EPSRC Science and Innovation award
to the Oxford Centre for Nonlinear PDE (EP/E035027/1).
The research of Zhongmin Qian was supported in part by the ERC grant (ESig ID291244).

\end{document}